\documentclass[10pt]{amsart}

\usepackage{latex_base}
\usepackage{macros}
\title[The tropical non-properness set of a polynomial map]{The tropical non-properness set of a polynomial map}
\author{Boulos El Hilany}
\thanks{For this work, the author was supported by the DFG Walter Benjamin Programme EL 1092/1-1
}
\thanks{MSC: Primary 14D06, Secondary: 14T20, 14M25, 58K15}
\thanks{Key words: Dominant polynomial maps, Jelonek set, Tropical Geometry, Newton Polytopes, Polyhedral Subdivisions}

\begin{document}

\maketitle
\begin{abstract} 

We study some discrete invariants of Newton non-degenerate polynomial maps $f : \mathbb{K}^n \to \mathbb{K}^n$ 
defined over an algebraically closed field of Puiseux series $\mathbb{K}$, equipped with a non-trivial valuation. 
It is known that the set $\mathcal{S}(f)$ of points at which $f$ is not finite forms an algebraic hypersurface in $\mathbb{K}^n$. 
The coordinate-wise valuation of $\mathcal{S}(f)\cap (\mathbb{K}^*)^n$ is a piecewise-linear object in $\mathbb{R}^n$, which we call the tropical 
non-properness set of $f$. We show that the tropical polynomial map corresponding to $f$ has fibers 
satisfying a particular combinatorial degeneracy condition exactly over points in the tropical non-properness set of $f$.
We then use this description to outline a polyhedral method for computing this set, and to recover the fan dual to the Newton polytope of the set at which a complex polynomial map is not finite. 
The proofs rely on classical correspondence and structural results from tropical geometry, combined with a new description of $\mathcal{S}(f)$ in terms of multivariate resultants.
\end{abstract}
 \markleft{}
 \markright{}
\section{Introduction}\label{sec:intro} A polynomial map $f:=(f_1,\ldots,f_n):\C^n\to \C^n$ is \emph{not proper} at $y\in\C^n$ if there is a sequence $\{x_k\}_k\subset \C^n$ satisfying $\Vert x_k\Vert\rightarrow \infty$ and $f(x_k)\rightarrow y$. Investigating the set $\cS(f)$ of points for which $f$ is not proper was initiated by Jelonek in~\cite{Jel93}, where he described its topological properties and provided computational results. He also used $\cS(f)$ to solve classical problems on the topology of $f$~\cite{Jel99,Jel99a,jelonek2005effective,Jel10}.  Currently, describing the geometry and topology of the \emph{Jelonek set}, $\cS(f)$, can be done exclusively using standard elimination methods~\cite[Proposition 7]{Jel93,Sta02}. However, these can often be inefficient whenever the polynomials involved have high degrees. 

Our main result is forward Theorem~\ref{th:main_corresp}. We establish in it a correspondence relating the ``non-properness'' of polynomial maps over a generalized field of Puiseux series to the so-called ``tropical non-properness'' of certain piecewise-affine maps. We call them \emph{tropical polynomial maps}. Tropical non-properness occurs whenever a preimage under a tropical polynomial map satisfies a certain non-degeneracy condition. We introduce in forward Theorem~\ref{th:main_computation} a recipe for recovering the tropical non-properness set by means of polyhedral geometry. Consequently, our results can be applied to reconstruct (see Remark~\ref{rem:dual-fan}) the dual fan of the Newton polytope corresponding to $\cS(f)$ above.

\subsection{Theoretical framework and main results}\label{sub:int:theoretical-frame} 
A polynomial map $f:X\to Y$, between two affine varieties over an algebraically closed field $K$ of characteristic zero, is said to be \emph{finite} if $K[X]$ is integral over $K[Y]$. In particular, the number of preimages (counted with multiplicities) in $X$ remains constant over any point in $Y$~\cite[Ch. I, \S 3]{shafarevich1994basic}. We say that $f$ is \emph{not finite at} $y\in Y$ if there is no Zariski open neighborhood $U$ of $y$ for which $f_{|f^{-1}(U)}:f^{-1}(U)\to U$ is finite. Note that a map $\C^n\to\C^n$ is not finite exactly at points in its Jelonek set (see e.g.,~\cite{jelonek2001topological,Sta02}). Jelonek's classical results~\cite{Jel93} were later extended to polynomial maps over any above $K$ as well as over the real numbers~\cite{Jel02,Sta02,Sta05,Tha09,JelLas18}. 
 Accordingly, the set of points in $X$ over which the map $f$ above is not finite will also be referred to as the \emph{Jelonek set}.

Let $\Bbbk$ be an algebraically closed field of characteristic zero. In this paper, we consider one-parametric series $c(t):=c_0t^{r_0}+c_1t^{r_1}+\cdots$ with coefficients in $\Bbbk$ and real exponents $r_0<r_1<\cdots$. These form an algebraically closed field $\K$ of characteristic zero~\cite{markwig2009field}. 
It is endowed with a function $\val:\K\to\R\cup\{-\infty\}$ sending zero to $-\infty$ and a non-zero element $c(t)$ above to the real value $-r_0$. The image $\Val(Z)$ under 
$\Val:=(\val(\cdot),\ldots,\val(\cdot))$ of an algebraic hypersurface $Z$ over $\K$, is a finite rational polyhedral complex of pure dimension $n-1$~\cite{groves1984geometry}. We call this image the \emph{tropicalization of $Z$}.  
It is known (see e.g.,~\cite{BIMS15,MS15}) that these objects, called \emph{tropical hypersurfaces}, preserve important topological criteria and other geometric invariants of their classical counterparts in the form of combinatorial and polyhedral data. 

The first correspondence theorem in tropical geometry was shown in~\cite{MKL06}. It stated that the tropicalization of an algebraic hypersurface $\{P = 0\}\subset (\K^*)^n$ is the locus in $\R^n$ at which the \emph{tropical polynomial} of $P$ is not differentiable.
 A tropical polynomial is a function $\R^n\to\R$ of the form $\max(L_1,\ldots,L_r)$, where each $L_1,\ldots,L_r$ is affine with positive integer coefficients.

 Tropical polynomials can be constructed from (classical) polynomials $\K^n\to\K$
 by replacing the coefficients with their valuations and the addition/multiplication operations with their respective tropical counterparts. The latter are defined for any values $a,b\in\R$, as $\max(a,b)$ and $a+b$ respectively. In this fashion, any tropical polynomial map $f^{\trop}:\R^n\to\R^m$ can be obtained from a polynomial map $f:\K^n\to\K^m$ by replacing the polynomials with their tropical counterparts. We call $f^{\trop}$ the \emph{tropical map of} $f$. For example, the below two tropical polynomials~\eqref{eq:main_example1_trop} constitute tropical map of~\eqref{eq:main_example1}: 
\begin{equation}\label{eq:main_example1_trop}
(F_1,~F_2):=(\max (x_2, ~ x_1+x_2,~ 2x_1+2x_2),~ \max (2x_2,~-2+x_1+x_2,~-4+2x_1+2x_2))
\end{equation}  
\begin{equation}\label{eq:main_example1}
(f_1,~f_2) := (z_2-z_1z_2 +2z_1^2z_2^2,~z_2^2 + t^{2}~z_1z_2 +t^{4}~z_1^2z_2^2)
\end{equation}

The \emph{virtual preimage} of a value $y\in\R$, under a tropical polynomial function $F: \R^n\to\R$, $x\mapsto\max(L_1(x),\ldots,L_r(x))$, is the set of all points in $\R^n$ at which the function $\max(L_1,\ldots,L_r,y)$ is not differentiable. We use the notation $\cT_{y}(F)$ to denote the virtual preimage of $y$ under $F$. In general, the virtual preimage is different from the (standard) preimage $F^{-1}(y)$, and can be viewed as the tropicalization of the preimage of a polynomial function. 
If, instead, we take $F$ to be a tropical polynomial map $(F_1,\ldots,F_n):\R^n\to\R^n$, then the \emph{virtual preimage} of a point $y\in\R^n$, also written as $\cT_y(F)$, is defined to be the intersection of all virtual preimages of $y_1,\ldots,y_n$, under the respective functions $F_1,\ldots,F_n$. Figure~\ref{fig:main-example} illustrates an example of the virtual preimage of $(-3,-5)$ under $F$ from~\eqref{eq:main_example1_trop}; it contains the union of two orange dots with the half-line (represented in the left-most side) whereas the standard preimage $F^{-1}(-3,-5)$ is the half-line.\\

\begin{definition}[Tropical non-properness]\label{def:trop-non-properness}
Let $H\subset\R^n$ be a half-line, and let $\alpha:=(\alpha_1,\ldots,\alpha_n)\in\R^n$ be a vector such that $H$ is a translation of $\alpha\cdot\R_{\geq 0}$. Then $H$ is said to be \emph{dictritical} if there exists $i,j\in\{1,\ldots,n\}$ such that $\alpha_i>0$ and $\alpha_j\leq 0$. The \emph{tropical non-properness set} $\TJF $ of a tropical polynomial map $F:\R^n\to\R^n$ is the set of all $y\in\R^n$ at which the virtual preimage contains a dicritical half-line.
\end{definition} 
Our paper concerns large families of polynomial maps $\KtK$, and we consider the subset of their Jelonek set contained in the torus $\bT:=(\K ^*)^n$. 
The \emph{support} of a polynomial is the set of exponent vectors of its monomials in $\N^n$ appearing  with non-zero coefficients. Given a tuple of polynomials $(f_1,\ldots,f_n)$, its \emph{support} is the tuple $A:=(A_1,\ldots,A_n)$ formed by the respective supports of the corresponding polynomials. 
All polynomials sharing a support $\sigma$ form a configuration space $\K^{\#\sigma}$ of points, whose coordinates represent the coefficients. In this way, we use $\K^A$ to denote the space of polynomial maps $\KtK$ supported on $A$, where
\[
\K^A \cong\K^{\# A_1}\times \cdots\times\K^{\# A_n}.
\] 

\begin{theorem}\label{th:main_corresp}
Let $A$ be any tuple of $n$ sets in $\N^n\setminus\{ (0,\ldots,0)\} $. Then, there is a Zariski open subset $\Omega\subset \K^A$  whose elements $f\in \Omega$ represent polynomial maps $\KtK$, such that 
\[
\Val(\cS(f)\cap\bT) = \TJF,
\] where $F:\R^n\to\R^n$ is the tropical polynomial map of $f$ (see e.g.,~\eqref{eq:main_example1} and~\eqref{eq:main_example1_trop}).
\end{theorem}

\begin{remark}\label{rem:origin}
The restrictions imposed in Theorem~\ref{th:main_corresp} are mild; on the one hand, since $\cS(f)$ is a hypersurface~\cite{Jel93}, computing its missing part at the coordinate hyperplanes is a trivial task. On the other hand, removing the origin from the supports equates to adding constants to the initial polynomials, which only translates 
the problem.

\end{remark}
We use the term \emph{face-generic} in reference to the maps $\KtK$ in $\Omega$ appearing in Theorem~\ref{th:main_corresp}. The latter theorem states that, in order to compute the valuation of the Jelonek set of a face-generic map it is enough to use the data of its tropical polynomial map. As a tropical hypersurface, this valuation is a rational piecewise-polyhedral complex, and we will present a combinatorial recipe for constructing it.

\begin{figure}[h]
\center
\input{fig-main-example}

\caption{
{\small Local illustrations of the virtual preimage under $F$ (taken from Example~\ref{ex:main:dim=2}) of some points in the set $\Val(\cS(f))$. The blue and black graphs correspond to the virtual preimages of the different points under $F_1$ and $F_2$.} 
}\label{fig:main-example}
\end{figure}

\begin{example}\label{ex:main:dim=2}
Consider the polynomial map~\eqref{eq:main_example1}, whose corresponding tropical polynomial map is~\eqref{eq:main_example1_trop}. In the coordinates $(w_1,w_2)$, its Jelonek set is given by the polynomial 
\begin{equation}\label{eq:example-Jelonek-set}
t^8~w_1^2 -(2t^4+ t^6)~w_1 - 4t^4~w_1w_2 - (2t^2 +t^4)~w_2 + 4~w_2^2
\end{equation}
 It was obtained using Gr\"obner basis methods~\cite{Jel93}.
Its tropical curve $\Val(\cS(f))$ is shown in purple in Figure~\ref{fig:main-example}; we obtained it using the tropical polynomial of~\eqref{eq:example-Jelonek-set}, expressed as
\[
\max( -8+2x_1,~ -4+x_1,~ -4+x_1+x_2, - 2+x_2,~2x_2).
\] For each $y\in\{ (-3,-5),~(-3,-2),~y_0\}\subset \Val(\cS(f))$, the virtual preimage (in orange) $\cT_y(F)$ contains a dicritical half-line spanned by the vector $(1,-1)$. \\
\end{example}
\begin{remark}\label{rem:dual-fan}
Results in~\cite{Jel93} and in~\cite{Sta02} provide a Gr\"obner-basis method to compute the polynomial $\mathsf{S}_f$ determining $\cS(f)$ for maps defined over any algebraically closed field $K$. The coefficients of $\mathsf{S}_f$ are polynomials over $\Z$ in the coefficients of $f$. This shows that, for generic complex maps $g:\C^n\to\C^n$, supported on $A$, the Newton polytope $\Delta\subset\R^n$ of $\cS(g)$ coincides with the Newton polytope of $\cS(f)$ for a face-generic map in $f\in\K^A$. On the other hand, a tropical hypersurface determines the dual fan of its Newton polytope. Therefore, Theorem~\ref{th:main_corresp} shows that one can recover the dual fan of $\Delta$ solely from the data of $\TJF$ corresponding to the tropical map of $f$. 
\end{remark}

\begin{example}\label{ex:main2:dim=2}[Example~\ref{ex:main:dim=2} ctnd.]
Let $B$ denote the pair of supports for the map in~\eqref{eq:main_example1}. Replacing the list $(1,-1,2,1,t^2,t^4)\in\K^B$ of its coefficients by a generic list of values $ \underline{b}:=(b_1,b_2,b_3,b'_1,b'_2,b'_3)\in\C^B$, we get a complex polynomial map $g\in\C^B$ satisfying
\begin{equation}\label{eq:Sg}
\mathsf{S}_g=
{b'_3}^2~w_1^2 +(b_2b'_2b'_3 - b_3{b'_2}^2 )~w_1 - 2b_3b'_3~w_1w_2 + (b_2b_3{b'}_2 - b_2^2{b'}_3)~w_2 + b_3^2 ~w_2^2.
\end{equation} Then, for a generic choice of $\underline{b}$, the Newton polytope of $\mathsf{S}_g$ is the convex hull of the set $\left(
\begin{smallmatrix}
  1 &  0 & 2 & 0 \\
  0 &  1 & 0 & 2 \\
\end{smallmatrix}
\right)$, which coincides with $\Delta$ from Remark~\ref{rem:dual-fan}. 
Furthermore, one can estimate $\Delta$ using its outer dual fan $\mathcal{F}\subset(\R^2)^*$, together with Jelonek's results in~\cite{Jel93} on the degree of the non-properness set. Indeed, on the one hand, the set $\TJF$ determines that the rays of $\mathcal{F}$ are generated by $(-1,-1)$, $(1,1)$, $(-1,0)$, and $(0,-1)$. On the other hand, using~\cite[Theorem 15]{Jel93}, we compute $\deg \mathsf{S}_g\leq (4\times 4 - 4)/4 = 3$. We can thus conclude that $\Delta$ is the convex hull of one of the three subsets $\left(
\begin{smallmatrix}
  1 &  0 & 2 & 0 \\
  0 &  1 & 0 & 2 \\
\end{smallmatrix}
\right)$, $\left(
\begin{smallmatrix}
  1 &  0 & 3 & 0 \\
  0 &  1 & 0 & 3 \\
\end{smallmatrix}
\right)$ or $\left(
\begin{smallmatrix}
  2 &  0 & 3 & 0 \\
  0 &  2 & 0 & 3 \\
\end{smallmatrix}
\right).$ 
\end{example}

A tropical polynomial map $F:\R^n\to\R^n$ induces a polyhedral cell-decomposition 
\[
\R^n=\bigsqcup_{\xi\in\Xi}\xi,
\] where every element in $\Xi$ is the relative interior of a convex polyhedron in $\R^n$, and the restriction $F_{|\xi}$ at each $\xi\subset \R^n$ is an affine map such that if $F_{|\xi}= F_{|\xi'}$, then either $\xi =\xi'$ or its closure $\overline{\xi}$ is a face of $\overline{\xi'}$ (c.f~\cite{BB13,grigoriev2022tropical}). For example, the map~\eqref{eq:main_example1_trop} gives rise to the decompsition of $\R^2$ represented in Figure~\ref{fig:example-th-comput} on the left. To each cell $\xi\in\Xi$ one associates the set $\mathcal{V}_\xi(F)$ consisting of points $y\in\R^n$ for which the virtual preimage $\cT_y(F)$ intersects the cell $\xi$. We use $\overline{\mathcal{V}_\xi(F)}$ to denote the closure (in the Euclidean topology) of $\mathcal{V}_\xi(F)$.\\

\begin{theorem}\label{th:main_computation}
Let $F:\R^n\to\R^n$ be the tropical map of a face-generic polynomial map $\KtK$ and let $\Xi$ denote the polyhedral cell-decomposition of $\R^n$ induced by $F$. Then, there exists a subset $\Upsilon\subset\Xi$, satisfying
\[
\TJF = \bigcup_{\xi\in\Upsilon} \mathcal{V}_\xi(F),
\] where each $\overline{\mathcal{V}_\xi(F)}$ is a polyhedron.
\end{theorem}  Elements $\xi$ from the subset $\Upsilon$ appearing in Theorem~\ref{th:main_computation} will be called \emph{contributing cells}. Along with the proof of Theorem~\ref{th:main_computation}, we provide in~\S\ref{sec:computing} a method for distinguishing contributing cells from the remaining ones in $\Xi$. In turn, this gives rise to an effective method for computing a polyhedral decomposition of $\TJF$ (see e.g., Figure~\ref{fig:example-th-comput}).

\begin{figure}[h]
\centering
\tikzset{every picture/.style={line width=1pt}} 

\begin{tikzpicture}[x=0.75pt,y=0.75pt,yscale=-1.25,xscale=1.25]

\draw [color={rgb, 255:red, 0; green, 0; blue, 0 }  ,draw opacity=1 ]   (114.08,179.97) -- (142.53,208.42) ;
\draw [color={rgb, 255:red, 74; green, 144; blue, 226 }  ,draw opacity=1 ]   (130.93,235.25) -- (76.08,179.97) ;
\draw [color={rgb, 255:red, 0; green, 0; blue, 0 }  ,draw opacity=1 ]   (114.08,179.97) -- (114.08,147.72) ;
\draw [color={rgb, 255:red, 74; green, 144; blue, 226 }  ,draw opacity=1 ]   (59.83,147.47) -- (76.08,179.97) ;
\draw [color={rgb, 255:red, 74; green, 144; blue, 226 }  ,draw opacity=1 ]   (76.08,179.83) -- (76.08,237.19) ;
\draw [color={rgb, 255:red, 0; green, 0; blue, 0 }  ,draw opacity=1 ]   (56.68,237.37) -- (113.72,180.33) ;
\draw  [color={rgb, 255:red, 245; green, 166; blue, 35 }  ,draw opacity=1 ][fill={rgb, 255:red, 245; green, 166; blue, 35 }  ,fill opacity=1 ] (74.4,218) .. controls (74.4,217.07) and (75.15,216.32) .. (76.08,216.32) .. controls (77.01,216.32) and (77.77,217.07) .. (77.77,218) .. controls (77.77,218.93) and (77.01,219.68) .. (76.08,219.68) .. controls (75.15,219.68) and (74.4,218.93) .. (74.4,218) -- cycle ;
\draw  [color={rgb, 255:red, 74; green, 144; blue, 226 }  ,draw opacity=1 ][fill={rgb, 255:red, 74; green, 144; blue, 226 }  ,fill opacity=1 ] (74.84,179.97) .. controls (74.84,179.28) and (75.4,178.73) .. (76.08,178.73) .. controls (76.77,178.73) and (77.32,179.28) .. (77.32,179.97) .. controls (77.32,180.65) and (76.77,181.21) .. (76.08,181.21) .. controls (75.4,181.21) and (74.84,180.65) .. (74.84,179.97) -- cycle ;
\draw [color={rgb, 255:red, 144; green, 19; blue, 254 }  ,draw opacity=1 ][fill={rgb, 255:red, 144; green, 19; blue, 254 }  ,fill opacity=1 ]   (212.47,192.98) -- (256.2,192.98) ;
\draw [color={rgb, 255:red, 144; green, 19; blue, 254 }  ,draw opacity=1 ][fill={rgb, 255:red, 144; green, 19; blue, 254 }  ,fill opacity=1 ]   (212.94,236.24) -- (256.2,192.98) ;
\draw [color={rgb, 255:red, 144; green, 19; blue, 254 }  ,draw opacity=1 ][fill={rgb, 255:red, 144; green, 19; blue, 254 }  ,fill opacity=1 ]   (256.2,192.98) -- (299.67,171.24) ;
\draw [color={rgb, 255:red, 144; green, 19; blue, 254 }  ,draw opacity=1 ][fill={rgb, 255:red, 144; green, 19; blue, 254 }  ,fill opacity=1 ]   (299.67,171.24) -- (299.67,235.71) ;
\draw [color={rgb, 255:red, 144; green, 19; blue, 254 }  ,draw opacity=1 ][fill={rgb, 255:red, 144; green, 19; blue, 254 }  ,fill opacity=1 ]   (321.4,149.51) -- (299.67,171.24) ;
\draw  [color={rgb, 255:red, 144; green, 19; blue, 254 }  ,draw opacity=1 ][fill={rgb, 255:red, 144; green, 19; blue, 254 }  ,fill opacity=1 ] (255.64,192.98) .. controls (255.64,192.67) and (255.89,192.41) .. (256.2,192.41) .. controls (256.52,192.41) and (256.77,192.67) .. (256.77,192.98) .. controls (256.77,193.29) and (256.52,193.54) .. (256.2,193.54) .. controls (255.89,193.54) and (255.64,193.29) .. (255.64,192.98) -- cycle ;
\draw  [color={rgb, 255:red, 144; green, 19; blue, 254 }  ,draw opacity=1 ][fill={rgb, 255:red, 144; green, 19; blue, 254 }  ,fill opacity=1 ] (299.11,171.24) .. controls (299.11,170.93) and (299.36,170.68) .. (299.67,170.68) .. controls (299.98,170.68) and (300.23,170.93) .. (300.23,171.24) .. controls (300.23,171.56) and (299.98,171.81) .. (299.67,171.81) .. controls (299.36,171.81) and (299.11,171.56) .. (299.11,171.24) -- cycle ;
\draw [color={rgb, 255:red, 144; green, 19; blue, 254 }  ,draw opacity=1 ][line width=1.5]    (133.07,237.19) -- (96.08,199.97) ;
\draw  [color={rgb, 255:red, 245; green, 166; blue, 35 }  ,draw opacity=1 ][fill={rgb, 255:red, 245; green, 166; blue, 35 }  ,fill opacity=1 ] (92.78,198.87) .. controls (92.78,197.94) and (93.53,197.18) .. (94.46,197.18) .. controls (95.39,197.18) and (96.14,197.94) .. (96.14,198.87) .. controls (96.14,199.8) and (95.39,200.55) .. (94.46,200.55) .. controls (93.53,200.55) and (92.78,199.8) .. (92.78,198.87) -- cycle ;
\draw [color={rgb, 255:red, 144; green, 19; blue, 254 }  ,draw opacity=1 ][line width=1.5]    (155.79,221.68) -- (114.08,179.97) ;
\draw  [color={rgb, 255:red, 0; green, 0; blue, 0 }  ,draw opacity=1 ][fill={rgb, 255:red, 0; green, 0; blue, 0 }  ,fill opacity=1 ] (112.84,179.97) .. controls (112.84,179.28) and (113.4,178.73) .. (114.08,178.73) .. controls (114.77,178.73) and (115.32,179.28) .. (115.32,179.97) .. controls (115.32,180.65) and (114.77,181.21) .. (114.08,181.21) .. controls (113.4,181.21) and (112.84,180.65) .. (112.84,179.97) -- cycle ;

\draw (106.59,254.67) node [anchor=north] [inner sep=0.75pt]  [font=\small]  {$\Xi $};
\draw (264.94,178.71) node [anchor=south] [inner sep=0.75pt]  [font=\tiny]  {$\mathcal{V}_b(F) $};
\draw (270.19,254.9) node [anchor=north] [inner sep=0.75pt]  [font=\small]  {$\Val(\cS(f))$};
\draw (72.84,179.97) node [anchor=east] [inner sep=0.75pt]  [font=\tiny]  {$( 0,0)$};
\draw (111.72,176.93) node [anchor=south east] [inner sep=0.75pt]  [font=\tiny]  {$( 2,0)$};
\draw (293.54,156.98) node [anchor=south] [inner sep=0.75pt]  [font=\tiny]  {$\mathcal{V}_c(F) $};
\draw (240.57,240.21) node [anchor=south] [inner sep=0.75pt]  [font=\tiny]  {$\mathcal{V}_a(F)$};
\draw (234.34,193.58) node [anchor=south] [inner sep=0.75pt]  [font=\tiny]  {$\mathcal{V}_\alpha(F)$};
\draw (301.67,203.48) node [anchor=west] [inner sep=0.75pt]  [font=\tiny]  {$\mathcal{V}_\beta(F)  $};
\draw (99.77,224.81) node [anchor=south] [inner sep=0.75pt]  [font=\small]  {$a$};
\draw (126.34,211.51) node [anchor=south] [inner sep=0.75pt]  [font=\small]  {$b$};
\draw (137.67,174.58) node [anchor=south] [inner sep=0.75pt]  [font=\small]  {$c$};
\draw (157.79,225.08) node [anchor=north west][inner sep=0.75pt]  [font=\small]  {$\beta $};
\draw (132.93,238.65) node [anchor=north west][inner sep=0.75pt]  [font=\small]  {$\alpha$};

\end{tikzpicture}

\caption{
{\small\textit{Left}: The decomposition induced by the map from Example~\ref{ex:main:dim=2}. The letters $a$, $b$ and $c$ are labels for the two-dimensional polyhedra containing them. The letters $\alpha$ and $\beta$ are the labels for the adjacent one-dimensional polyhedra (in purple). \textit{Right}: The tropical non-properness set is made up of edges $\mathcal{V}_\xi(F)$, where $\xi\in\{a,b,c,\alpha,\beta\}$; for each point $y$ in $\mathcal{V}_\xi(F)$, its virtual preimage has a dicritical half-line in the cell of $\Xi$ with the label $\xi$.}
}
\label{fig:example-th-comput}

\end{figure}

\subsection{Organization of the paper}\label{sub:int:main-res} 
\S\ref{sec:faces} is devoted to introducing some of the necessary conditions for a polynomial map to be face-generic, called \emph{almost face-generic} (Definition~\ref{def:face-generic}). These conditions apply for any polynomial map over an algebraically-closed field $K$ of characteristic zero. We show in Proposition~\ref{prp:face-generic-open} that almost face-genericity is an open condition in $K^A$. This will be useful for describing the Jelonek set as a union of \emph{face-resultants} in $K^n$ (Definition~\ref{def:r-star-gamma} and Theorem~\ref{thm:algebraic_correspondence}). Face-resultants are obtained from square polynomial systems by restricting individual polynomials of the map onto faces of their respective Newton polytopes. A subset $S\subset K^n$ is said to be \emph{$K$-uniruled} if for every $y\in S$ there exists a polynomial map $\varphi:K\to S$ with $\varphi(0)=y$. Thanks to our description in~\S\ref{sec:prel} of the Jelonek set, we re-prove Jelonek's result on $K$-uniruledness for almost face-generic maps over $K$ (Corollary~\ref{cor:k-uniruled}). 

In \S\ref{sec:prel} we introduce notations and some classical results from basic tropical algebraic geometry. One ingredient needed from~\S\ref{sec:prel} is the classical tropical correspondence theorem for transversal intersections (Theorem~\ref{thm:Kap-general}), and another one is the duality between the tropical hypersurfaces and subdivisions of polytopes (Theorem~\ref{thm:mixed_subd}). For Theorem~\ref{th:main_corresp}, we show in~\S\ref{sec:correspondence} that the tropical non-properness set is the union of valuations of the face-resultant sets (Theorem~\ref{thm:tropical-and-tilted}). This, together with results in Section~\ref{sec:prel}, yields Theorem~\ref{th:main_corresp}. The proof of Theorem~\ref{th:main_computation} is in~\S\ref{sec:computing}.

\section{Jelonek set from faces of polytopes}\label{sec:faces} 

We start this section with some notations and definitions. In \S\ref{sub:prelim} we define almost face-generic polynomial maps, and show that they form a Zariski open subset in the space of all maps with the same support. In~\S\ref{sub:pol-restricted}, we describe some relevant properties of face-resultants for almost face-generic maps. We recall in~\S\ref{sub:Bernstein} Bernstein's theorem for polynomial tuples and apply it to determine the topological degree of a polynomial map. The correspondence between face-resultants and the Jelonek set culminates in~\S\ref{sub:non-prop_resultants}. We finish this section by using this description to re-prove $K$-uniruledness for the Jelonek set.

\subsection{Preliminaries on Newton polytopes}\label{sub:prelim} A subset of $\R^n$ is called a \emph{polyhedron} if it is the intersection of finitely-many closed half-spaces. A \emph{face} $\Phi$ of a polyhedron $\Pi\subset\R^n$ is given by the intersection of $\Pi$ with a hyperplane $H$ for which $\Pi$ is contained in the closure of one of the two halfspaces determined by $H$. We call $H$ a \emph{supporting} hyperplane of $\Phi$. The face $\Phi$ is called \emph{proper} if $\Phi\neq\Pi$. We say that $\Phi$ is \emph{dicritical} if the outward normal vector of one of its supporting hyperplanes positively spans a dicritical half-line. That is, there exists a functional $(x_1,\ldots,x_n)\mapsto \alpha_1x_1 +\cdots + \alpha_nx_n$ that maximizes $\Pi$ at $\Phi$, and satisfies $\alpha_i>0$ and $\alpha_j\leq 0$ for some $i,j\in\{1,\ldots,n\}$.

If a polyhedron is bounded, then we call it a \emph{polytope}. Given a tuple of polytopes $\Pi:=(\Pi_1,\ldots,\Pi_n)$, the Minkowski sum of their members is a polytope in $\R^n$, denoted by $\sum\Pi$. Here, the \emph{Minkowski sum} of any two subsets $A,B\subset\R^n$ is the coordinate-wise sum $
A+B:=\{a+b~|~a\in A,~b\in B\}$. For every integer $n\in\N$, we use $[n]$  as a shorthand for $\{1,\ldots,n\}$. The sub-tuple $(\Pi_i)_{i\in I}$ is denoted by $I\Pi$ for any $I\subset [n]$. A tuple $\Gamma:=(\Gamma_1,\ldots,\Gamma_n)$ of polytopes in $\R^n$ is said to be a \emph{tuple-face} (or, simply, \emph{face} whenever it is clear from the context) of $\Pi$ if the subset $\Gamma_i$ is a face of $\Pi_i$ ($i=1,\ldots,n$) and $\sum\Gamma$ is a face of $\sum\Pi$. We use the notation $\Gamma\prec\Pi$. We say that $\Gamma\prec\Pi$ is a \emph{proper} face if $\sum\Gamma$ a proper face of $\sum\Pi$.\\ 

\begin{notation}\label{not:origin}
We use $\unze$ to denote the point $(0,\ldots,0)$, and $\Rnn$ to denote the subset in  $\R^n$ formed by all points with non-negative coordinates. 
\end{notation} 

\begin{definition}\label{def:dicritical-tupe}
For any polytope $\Pi\subset\Rnn$, either $\unze$ is a vertex of $\Pi$, or $\unze\not\in \Pi$. A face $\Gamma$ of a tuple of polytopes $\Delta$ is called \emph{origin} if $\unze\in\sum\Gamma$. This is equivalent to $\Gamma_i$ containing $\unze$ for all $i\in [n]$. We say that $\Gamma$ is \emph{semi-origin} if $\unze\in\Gamma_i$ for some $i\in [n]$, and \emph{dicritical} if $\sum\Gamma$ is a dicritical face of $\sum\Delta$.
\end{definition}

\begin{figure}
\centering
\input{fig-three-polytopes}
\caption{A triple of polytopes $\Delta=(\Delta_1,\Delta_2,\Delta_3)$ in $\R^3$ obtained from Example~\ref{ex:rgamma}.}\label{fig:three_polytopes}
\end{figure}

\begin{example}\label{exp:possibly-origin}
Figure~\ref{fig:three_polytopes} represents a triple $\Delta$ of polytopes in $\R^3$. The face traced in red is dicritical and not origin. So is the face traced in green. The face obtained by intersecting $\sum\Delta$ with the horizontal coordinate plane is origin and not dicritical.
\end{example} 

Let $K$ be an algebraically closed field of characteristic zero, and let $f$ be a polynomial in $K[z_1,\ldots,z_n]$. Then $f$ is expressed as a linear combination 
\[
\sum_{a\in A} c_a z^a
\]
of monomials $z^a:=z_1^{a_1}\cdots z_n^{a_n}$, where $A$ is a finite subset of $\N^n$, and $c_a\in{\!K}^*$.  We call $A$ the \emph{support} of $f$, and its \emph{Newton polytope} is the convex hull of $A$ in $\R^n$. If $\sigma$ is a subset of $A$, the \emph{restriction of $f$ to $\sigma$}, denoted by $f_{\sigma}$, is the polynomial $\sum_{a\in\sigma}c_az^a$. The notation $\VKa(f)$ will refer to the zero locus $\{z\in \bT~|~f(z)=0\}$. Now, if $f$ is a tuple of polynomials $f_1,\ldots,f_k\in K[z_1,\ldots,z_n]$, we use the same notation $\VKa(f)$ for the intersection $\VKa(f_1)\cap\cdots\cap \VKa(f_k)\subset \bT$.

\subsection{Face-resultants}\label{sub:pol-restricted}

Let $A:=(A_1,\ldots,A_n)$ be a tuple of subsets in $\N^n\setminus\{\unze\}$, and let $\Delta:=(\Delta_1,\ldots,\Delta_n)$ denote the tuple of polytopes in $\Rnn$ given by
\begin{equation}\label{eq:delta_0}
\Delta_i:=\conv\big(\{\unze\}\cup\! A_i\big)~\text{for } i =1,\ldots,n.
\end{equation} Recall that $K^A$ refers to the space of all tuples of polynomials supported on $A$, and let $f:=(f_1,\ldots,f_n)$ be a tuple in $K^A$ defining a polynomial map $f:K^n\to K^n$. We assume in what follows that $\sum A$ affinely spans $\R^n$.\\

\begin{notation}\label{not:face-restrictions}
Let $\Gamma\prec\Delta$. For $i=1,\ldots,n$, we use $f_{i\Gamma}$ to denote the restricted polynomial $(f_i)_{\Gamma_i\cap A_i}$, and for any $w:=(w_1,\ldots,w_n)\in K^n$, we define 
\begin{equation}\label{eq:cases:omega}
(f_i - w_i)_\Gamma:= \begin{cases}
f_{i\Gamma} - w_i & \text{if }  \{\unze\}\subset\Gamma_i \\[4pt]
f_{i\Gamma}  & \text{otherwise}. \\
\end{cases}
\end{equation} The tuples $(f_{1\Gamma},\ldots,f_{n\Gamma})$ and $((f_1 - w_1)_\Gamma,\ldots,(f_n - w_n)_\Gamma)$ are denoted by $f_\Gamma$ and $(f-w)_\Gamma$ respectively. For any subset $I\subset [n]$, we use the notations $
f_{I\Gamma}:=(f_{i\Gamma})_{i\in I}$ and $(f - w)_{I\Gamma}:=((f_i-w_i)_{i\Gamma})_{i\in I}$.
\end{notation} 

\begin{definition}\label{def:multivariate_resultant}
For any face $\Gamma\prec\Delta$, and any $I\subset [n]$, we define two different \emph{multivariate resultants} (c.f.~\cite{GKZ94,St94,Est13}) $\Res_{I\Gamma}^A$, and $\Res_{I\Gamma}^\Delta$, given by 
\begin{align*}
\Res_{I\Gamma}^A &:= \left\lbrace f\in K^A~\left|~f_{I\Gamma} = 0 \text{ has a solution in $\bT$}  \right\rbrace\right. ,\\
\Res_{I\Gamma}^\Delta &:= \left\lbrace (f,w)\in K^A\times K^n~\left|~(f-w)_{I\Gamma}=0 \text{ has a solution in $\bT$}  \right\rbrace\right. ,
\end{align*} 
\end{definition}

In what follows, we use classical results on the multivariate resultants to prove technical statements concerning the main theorem of this section.\\

\begin{definition}\label{def:face-generic}
The tuple $f$ is called \emph{almost face-generic} if for any $\Gamma\prec\Delta$ the following conditions hold:
\begin{enumerate}[topsep=0mm,itemsep=0mm,label=(\alph*),ref=(\alph*)]

	\item\label{it:vg_not_vgprime} for any $\Gamma'\prec\Delta$ different from $\Gamma$, the sets $\VKa(f_{\Gamma})$ and $\VKa(f_{\Gamma'})$ are equal only if both are empty,\\
	
	\item\label{it:vgi_smooth} for any $I\subset [n]$, the set $\VKa(f_{I\Gamma})$ is a smooth irreducible complete intersection satisfying 
		\begin{equation}\label{eq:def:Newton}
			\# I > \dim \sum I\Gamma \Longrightarrow \VKa(f_{I\Gamma})=\emptyset.
		\end{equation}
\end{enumerate}

\end{definition} 

\begin{proposition}\label{prp:face-generic-open}
The set of almost face-generic tuples in $K^A$ contains a Zariski open subset.
\end{proposition}

\begin{proof}
Let $\Omega_a$ and $\Omega_b$ be the two subsets consisting of tuples in $K^A$ satisfying conditions~\ref{it:vg_not_vgprime} and~\ref{it:vgi_smooth} respectively. 
It is enough to show that each of these sets is Zariski open in $K^A$. 

Concerning~\ref{it:vg_not_vgprime}: We have $\Gamma_i\neq \Gamma'_i$ for some $i\in [n]$. Let $a$ be any lattice point in $\Gamma_i\setminus\Gamma_i'$ (the case $a\in \Gamma_i'\setminus\Gamma_i$ is analogous). Then, for any $f\in K^A$, the monomial term $cz^a$ appears in the ideal $\langle f_\Gamma\rangle$ generating $\VKa(f_\Gamma)$, but not in $\langle f_{\Gamma'}\rangle$. Since $K$ has characteristic zero, for all but finitely-many $c\in K$, we have $\langle f_\Gamma\rangle\neq \langle f_{\Gamma'}\rangle$. This yields Item~\ref{it:vg_not_vgprime}.

Concerning~\ref{it:vgi_smooth}: Consider the projection $P_{I\Gamma}:Z_{I\Gamma}\subset \bT\times K^A\to K^A$, where $Z_{I\Gamma}:=\{(z,f)~|~f_{I\Gamma}=\unze \}$. Let $\Crit P_{I\Gamma} $ denote the set of all $(z,f)$ for which $\VKa(f_{I\Gamma})$ is not a smooth irreducible complete intersection at $z\in\bT$. Then, $K^A\setminus P_{I\Gamma}(\Crit P_{I\Gamma})$ is Zariski open in $K^A$ follows from Bertini-Sard Theorem.

Concerning Equation~\eqref{eq:def:Newton}: The system $f_{I\Gamma} = \unze$ having a solution in $\bT$, implies that $f$ lies in the multivariate resultant $\Res_{I\Gamma}^A\subset K^A$. This is an algebraic set in $K^A$, satisfying the following well-known equality (see also~\cite{St94,Est13})
\begin{equation}\label{eq:resultants}
 \cdim \Res_{I\Gamma}^A =  \max_{J\subset I}\big(\#J -\dim \sum_{i\in J}\Gamma_i\big).
\end{equation} Therefore, $\cdim\Res_{I\Gamma}^A>0$ if $I\subset [n]$ satisfies 
\begin{equation}\label{eq:inequality_resultant_codim}
\# I > \dim \sum I\Gamma.
\end{equation} Thus, it is enough to choose $\Omega_b\subset K^A$ to be the largest subset that avoids subsets $P_{I\Gamma}(\Crit P_{I\Gamma})$ for each $\Gamma\prec\Delta$ and each $I\subset [n]$, as well as resultants $\Res_{I\Gamma}^A$, for those $I$ satisfying~\eqref{eq:inequality_resultant_codim}. As subsets $\Res_{I\Gamma}$ are Zariski closed, and since there are finitely-many faces $I\Gamma\prec I\Delta$ of sub-tuples, the set $\Omega_b$ becomes Zariski open.
\end{proof}

\begin{definition}\label{def:r-star-gamma}
Let $\Gamma$ be a proper face of $\Delta$, and let $X_\Gamma$ denote the set  
\[
\left\{(z,w)\in \bT\times K^n~\left|~(f(z)-w)_\Gamma =0\right\}\right. .
\] The \emph{face-resultant} of the map $f$ at the face $\Gamma$, denoted by $\cR_{\Gamma}(f)$, is the closure in $K^n$ of the image of  $X_\Gamma$ under the projection $\pi:~(z,w)\mapsto w$.
\end{definition}

We illustrate face-resultants in the following example.\\

\begin{example}\label{ex:rgamma}
Let $f$ be the below triple of polynomials in three variables $a,b,c$, defined over the field $\K$ of generalized Puiseux series in $t$:
{\small
\begin{equation}\label{eq:main_example3}
\begin{array}{ll}
 f_1:= & abc + ~ bc +~a^2b^2c^2,\\[4pt]
 f_2:= & ab +~ abc^2+~t^7~bc^2+~t^3~b^2c^4+~t^2~ab^2c^4+~t^5~a^2b^2c^4,\\[4pt]
 f_3:= & t~a +~ ab+~abc+~a^2bc. 
\end{array}
\end{equation}}The triple of polytopes $\Delta$ in Figure~\ref{fig:three_polytopes} correspond to~\eqref{eq:main_example3}. Let $\Gamma\prec\Delta$ be the triple highlighted in green. That is, $\Gamma_i = \conv(S_i)$, where 
{\small 
		\[S_1:=\left(
\begin{smallmatrix}
  0 &  0 & 2 \\
  0 &  1 & 2 \\
  0 &  1 & 2 \\
\end{smallmatrix}
\right),\quad 
S_2:=\left(
\begin{smallmatrix}
  0 &  2  \\
  2 &  2 \\
  4 &  4 \\
\end{smallmatrix}
\right),\quad 
S_3:=\left(
\begin{smallmatrix}
  0 & 1 & 1 & 2 \\
  0 & 0 & 1 & 1 \\
  0 & 0 & 1 & 1 \\
\end{smallmatrix}
\right).
 \]
} Then $(f - w)_{\Gamma}$ is the triple
{\small
\begin{equation}\label{eq:main_example_restricted}
\begin{array}{lll}
 (f_1 - w_1)_{\Gamma}= & f_1 - w_1 &=  abc + ~ bc +~a^2b^2c^2 - ~~w_1,\\[4pt]
 (f_2 - w_2)_{\Gamma}= & f_{2\Gamma} &=  ~t^3~b^2c^4+~t^2~ab^2c^4+~t^5~a^2b^2c^4,\\[4pt]
 (f_3 - w_3)_{\Gamma}= & f_{3\Gamma} - w_3& =  t~a +~abc+~a^2bc -~~w_3,
\end{array}
\end{equation}
} and $\crsg(f)$ is the surface defined by the below polynomial in $\K[w_1,w_3]\subset \K[w_1,w_2,w_3]$:
{\small
\begin{multline*}
t^4( -2 +4t +t^3)+~t(1 -2t -4t^3 +6t^4+2t^6)~w_1  +~t^3(	-2 +3t+4 t^2+t^3)~w_3 \\
 +~t(1-2t +t^2 -2t^3 +2t^4 +t^6)~w_1^2 
 + (1 -2t -t^2 -2t^3 +8t^5 +t^6)~w_1w_3\\
  +~t^3( -1 +4 t +t^2 -2t^3 +4t^4 + t^6)~w_3^2   +~2t(-1 +t^3 +t^4 -t^6 )~w_1w_3^2 +~t^4(1 -t^2 +4t^3 )~w_3^3 +~t^7- ~~w_3^4.
\end{multline*} 
}
Now, if we pick $\Gamma\prec\Delta$ be the triple from Figure~\ref{fig:three_polytopes} highlighted in red, then $\Gamma_i = \conv(S_i)$, where 
{\small 
		\[S_1:=\left(
\begin{smallmatrix}
  0 &   2 \\
  0 &   2 \\
  0 &  2 \\
\end{smallmatrix}
\right),\quad 
S_2:=\left(
\begin{smallmatrix}
  0 & 1 & 2  \\
  0 & 1 & 2 \\
  0 & 0 & 4 \\
\end{smallmatrix}
\right),\quad 
S_3:=\left(
\begin{smallmatrix}
  1  & 2 \\
  0  & 1 \\
   0  & 1 \\
\end{smallmatrix}
\right).
 \]
} Then $\crsg(f)$ is the hyperplane defined by the polynomial $t^2 - t - w_1$.
\end{example}

 We finish this part by proving three technical properties about the face-resultants. Assume in what follows that $f$ is almost face-generic.\\

\begin{lemma}\label{lem:proj-codim>1}  
Let $\Gamma$ be a proper face of $\Delta$. Then, the face-resultant $\crsg(f)$ is empty if $\Gamma$ is not semi-origin (see Definition~\ref{def:dicritical-tupe}), and $\dim\crsg(f)\leq n-1$ otherwise. Furthermore, for any other proper face $\Gamma'$ of $\Delta$, it holds that $\dim\crsg(f)\cap\cR_{\Gamma'}(f)\leq n-2$.
\end{lemma}

\begin{proof}
Keeping the notations of Definition~\ref{def:r-star-gamma}, if $\Gamma$ is not semi-origin, then for any $w\in \K^n$, it holds that $(f-w)_\Gamma = f_\Gamma $. Equation~\eqref{eq:def:Newton} of Definition~\ref{def:face-generic} shows that $X_\Gamma=\emptyset$, and thus its projection under $\pi$ is empty. 

 Let $\Gamma$ be a semi-origin face. Now, we show that $\dim\crsg(f)\leq n-1$. Note that if $\dim X_\Gamma\leq n-1$, then $\dim \pi(X_\Gamma)\leq n-1$ and we are done. 
 
 Assume in what follows that $\Gamma$ is semi-origin, and that $\dim X_\Gamma \geq n$. Recall that $\VKa(f_{I\Gamma})$ is a complete intersection for every $I\subset[n]$ (Definition~\ref{def:face-generic}~\ref{it:vgi_smooth}). Then, it is an easy computation, using the Jacobian matrix $\Jac_{(z,w)} (f-w)_{\Gamma}$, to show of $X_\Gamma$ is either empty or is a complete intersection. We thus get $\dim X_\Gamma = n$. Since $\dim \sum\Gamma\leq n-1$, up to multiplying each polynomial by an appropriate monomial, for any  $w\in K^n$, the system $(f-w)_\Gamma =\unze$ is formed by quasi-homogeneous polynomials in $z$. Consequently, the fibers under the map $\pi_{|X_\Gamma}:X_\Gamma\to\pi(X_\Gamma) $ have positive dimension. Hence, we get $\dim X_\Gamma \geq \dim \pi(X_\Gamma) + d $ for some $d>0$, and thus $\dim\crsg(f)=\dim \pi(X_\Gamma)\leq n-1$. This yields the proof of the first part of the Lemma.

Now we show that $\dim\crsg(f)\cap\cR_{\Gamma'}(f)\leq n-2$. By contradiction, assume that $\crsg(f)$ and $\cR_{\Gamma'}(f)$ share an $(n-1)$-dimensional component. Since $f$ is almost face-generic, both $X_\Gamma$ and $X_{\Gamma'}$ are smooth, and thus irreducible. Consequently we obtain $\crsg(f) =\cR_{\Gamma'}(f)$. Recall that $\pi$ has fibers of positive dimension when restricted to either $X_\Gamma$ or $X_{\Gamma'}$. Therefore, the equality
\[
\dim\cR_\Gamma(f) =\dim\cR_{\Gamma'}(f) =n-1
\] implies that $X_\Gamma =X_{\Gamma'}$. This contradicts Definition~\ref{def:face-generic}~\ref{it:vg_not_vgprime} as $f$ is almost face-generic and $\Gamma\neq\Gamma'$. 
\end{proof}

For any $i\in[n]$, we use $H^{\R}_i$ and $H^{K}_i$ to denote the $i$-th coordinate hyperplane $\{x_i=0\}$ of $\R^n$ and of $K^n$ respectively. The following statement follows directly from the definitions of $f(H^{K}_i)$ and $\crsg(f)$.\\

\begin{lemma}\label{lem:resultant_disjoint_from_coordinate}
Assume that $\Delta$ has a face $\Gamma\prec\Delta$ satisfying $\sum\Gamma= \sum\Delta\cap H^{\R}_i$ for some $i\in [n]$. Then, it holds that $f(H^{K}_i) = \crsg(f)$. 

\end{lemma}
 Recall that a face $\Gamma\prec\Delta$ is called dicritical if one of the a supporting vectors of $\sum\Gamma$ has one positive coordinate and one non-positive coordinate (see~\S\ref{sub:prelim} and Definition~\ref{def:dicritical-tupe}).\\

\begin{lemma}\label{lem:face-inclusion_propertuy}
Let $\Gamma\prec\Delta$ be a dicritical face satisfying $\sum\Gamma\subset H_{i}^{\R}$ for some $i\in [n]$. Then, there exists another dicritical face $\Gamma'\prec\Delta$ (see e.g., Example~\ref{ex:face_in_face}) satisfying \textbf{(a)} $\Gamma\prec\Gamma'$, \textbf{(b)} $\sum\Gamma'\not\subset H_j^{\R}$ for all $j\in [n]$, and 
\begin{equation}\label{eq:inclusion_face-resultant}
\mathcal{R}_\Gamma(f)\subset\mathcal{R}_{\Gamma'}(f).
\end{equation}
\end{lemma}

\begin{proof}
Since $\Gamma$ is dicritical and $\sum\Gamma\subset H_i^{\R}$, we get that $\sum\Gamma$ contains the origin $\{\unze\}$, and its affine span $L_\Gamma:=\afff(\sum\Gamma)$ is a proper linear subspace of $H_i^{\R}$. Consider the orthogonal linear subspace $L_\Gamma^\perp\subset\R^n$ generated by the kernel of $L_\Gamma$. Then, the projection $\pi_{\Gamma}:\R^n\longrightarrow L_{\Gamma}^\perp$ sends $\sum\Delta$ to  the polytope $\Sigma:=\pi_{\Gamma}(\sum\Delta)\subset L_{\Gamma}^\perp\cong\R^\nu$, where $\nu:=\cdim L_\Gamma$. Furthermore, the face $\sum\Gamma$ is sent to the origin $\{\unze\}\subset L_\Gamma^\perp\subset\R^n$, and the latter is a vertex of $\tilde{\Delta}$.  Then, one can choose a face $F\subset\Sigma$ satisfying the following properties:\\
\begin{enumerate}

	\item $\dim F>0$,
	
	\item $\{\unze\}\subset F$,
	
	\item $F$ does not belong to any coordinate hyperplane of $\R^n$, and
	
	\item there exists an outer normal vector $\alpha_F\in L_\Gamma^{\perp}$ to $F$ satisfying the dicritical property (see Definition~\ref{def:dicritical-tupe}).\\
\end{enumerate}
The existence of $F$ above follows from $\{\unze\}$ being a vertex of $\Sigma$. Then, there exists a face $F'\subset\sum\Delta$ such that $F=\pi_\Gamma(F')$. Furthermore, we have $\alpha_F$ is an outer normal vector to $F'$. Consequently, we choose the face $\Gamma'\prec\Delta$ so that $F'=\sum\Gamma'$. We leave it to the reader to check that $\Gamma'$ is dicritical and properties \textbf{(a)} and \textbf{(b)} on $\Gamma'$ follow from the four criteria above on $F$.

Concerning relation~\eqref{eq:inclusion_face-resultant}, we use the well-known property on multivariate resultants (see, e.g.~\cite[Lemma 2.3, (3)]{Est13})
\begin{equation}\label{eq:lem-Esterov}
\Gamma\prec\Gamma'\Longrightarrow \Res_{\Gamma}^{\Delta}\subset \Res_{\Gamma'}^{\Delta}.
\end{equation} Note that, for a fixed map $f$, a face-resultant $\cR_\Gamma(f)$ is identified as the subset of points $(g,w)\in \Res_\Gamma^\Delta$ satisfying $g=f$. That is, it holds that 
\[
\cR_\Gamma(f)\cong\Res_\Gamma^\Delta\cap \left\lbrace (f,-w)\in K^A\times K^n~\left|~w\in K^n \right\rbrace \right. . 
\] Hence, Relation~\eqref{eq:inclusion_face-resultant} follows from Property \textbf{(a)} and~\eqref{eq:lem-Esterov}.
\end{proof}

\begin{example}\label{ex:face_in_face}
An instance for properties \textbf{a)} and \textbf{b)} of Lemma~\ref{lem:face-inclusion_propertuy} appear in Figure~\ref{fig:three_polytopes}; if $\Gamma\prec\Delta$ is the horizontal face satisfying $\Gamma_i = \conv(S_i)$, where 
{\small 
		\[S_1:=\left(
\begin{smallmatrix}
  0  \\
  0 \\
  0 \\
\end{smallmatrix}
\right),\quad 
S_2:=\left(
\begin{smallmatrix}
  0   \\
  0  \\
  0 \\
\end{smallmatrix}
\right),\quad 
S_3:=\left(
\begin{smallmatrix}
  0 & 1 \\
  0 & 0 \\
  0 & 0 \\
\end{smallmatrix}
\right).
 \] 
 } Then, the face $\Gamma'$ can be chosen so that 
 {\small 
		\[S'_1:=\left(
\begin{smallmatrix}
  0  \\
  0 \\
  0 \\
\end{smallmatrix}
\right),\quad 
S'_2:=\left(
\begin{smallmatrix}
  0 & 0 & 2  \\
  0 & 2 & 2 \\
  0 & 4 & 4\\
\end{smallmatrix}
\right),\quad 
S'_3:=\left(
\begin{smallmatrix}
  0 & 1 \\
  0 & 0 \\
  0 & 0 \\
\end{smallmatrix}
\right).
 \] 
 }
\end{example}

\subsection{Topological degree of a map}\label{sub:Bernstein} If $\Pi$ is a convex body in $\R^n$, we use $\Vol_n(\Pi)$ to denote its Euclidean volume. Assume now that $\Pi$ is a collection of $n$ convex bodies $\Pi_1,\ldots,\Pi_n\subset\R^n$. Then the classical notion of \emph{mixed volume} of $\Pi_1,\ldots,\Pi_n$ is defined as
\[
V(\Pi)=V(\Pi_1,\ldots,\Pi_n):=\frac{1}{n!}\sum_{k=1}^n(-1)^{k+n}\sum_{i_1<\cdots<i_k}\Vol_n(\Pi_{i_1}+\cdots+\Pi_{i_k}),
\] where the summation between polytopes is the Minkowski sum. For example, if $n=2$, we have 
\[
V(\Pi)=\Vol_2(\Pi_1+\Pi_2) - \Vol_2(\Pi_1) - \Vol_2(\Pi_2).
\]

We have the following well-known version of Bernstein's Theorem~\cite{Ber75} (see e.g.~\cite{rojas1999toric}).\\
\begin{theorem}\label{th:Bernstein}
Let $g_1,\ldots,g_n$ be polynomials in $n$ variables $K[z_1,\ldots,z_n]$, such that $V(\Pi_1,\ldots,\Pi_n)\neq 0$, where $\Pi_1,\ldots,\Pi_n$ denote their respective Newton polytopes in $\R^n$, and let $I^*(g)$ denote the number of isolated solutions in $\VKa(g)$, counted with multiplicities. Then, it holds that $I^*(g) \leq V(\Pi_1,\ldots,\Pi_n)$, with equality if and only if $g_\Gamma = \unze$ does not have solutions in $\bT$ for any proper face $\Gamma$ of $\Pi$.
\end{theorem}
\ \\
\begin{definition}\label{def:top-degree}
Let $f:K^n\to K^n$ be a polynomial map. We say that $f$ is \emph{dominant} if the determinant of its Jacobian matrix is a polynomial that is not identically zero. If $f$ is dominant, there is an integer $k\in\N$ and a Zariski open subset $W\subset K^n$ for which $f^{-1}(w)$ has exactly $k$ isolated points for any $w\in W$. We call the value $k$ the \emph{topological degree of $f$}, and we denote it by $\mu(f)$.
\end{definition}
\begin{remark}\label{rem:top-degree-non-prop}
The Jelonek set $\cS(f)$ of the map $f$ from Definition~\ref{def:top-degree} can be defined as the set of points $w\in K^n$ at which the polynomial system $f-w=\unze$ has less than $\mu(f)$ isolated solutions in $K^n$ counted with multiplicities.
\end{remark}

\ \\

\begin{lemma}\label{lem:preimages}
For any $f\in K^A$ representing a dominant almost face-generic map $K^n\to K^n$, we have $\mu(f)=V(\Delta)$, where $\Delta$ is defined in Equation~\eqref{eq:delta_0}.
\end{lemma}

\begin{proof}
Take $W=f(\bT)\cap (\cR\cup \cS(f))$, where $\cR$ is the union of all face-resultants of $f$. Recall that $f$ is dominant, and the Jelonek set is a hypersurface~\cite{Jel93}. These facts, together with Proposition~\ref{prp:face-generic-open}, imply that $W$ is Zariski open in $K^n$. Since $w\in f(\bT)$, we get $f^{-1}(w)\subset \bT$. By definition, the map is finite outside $\cS(f)$, and thus has a constant number of preimages (counted with multiplicities) $\mu(f)$ that is independent of $w\not\in\cS(f)$ (see e.g.~\S\ref{sub:int:main-res}). In the notation of Theorem~\ref{th:Bernstein}, this shows that $\mu(f) = I^*(f-w)$. Finally,  Theorem~\ref{th:Bernstein} shows that for any $w\in W$, it holds that $ I^*(f-w)= V(\Delta_1,\ldots,\Delta_n)$.
\end{proof}

\subsection{The Jelonek set and face-resultants}\label{sub:non-prop_resultants} We finish this section by proving that the Jelonek set of a map is the union of its face-resultants. \\

\begin{notation}\label{not:dicritical_set}
For any tuple $\Pi$ of polytopes in $\R^n$, we use $\mathscr{F}_{\!D}(\Pi)$ to denote the set of all its dicritical semi-origin faces.
\end{notation}

\begin{theorem}\label{thm:algebraic_correspondence}
Let $A:=(A_1,\ldots,A_n)$ be a collection of finite subsets in $\N^n\setminus\{\unze\}$, and let $\Delta:=(\Delta_1,\ldots,\Delta_n)$ denote the tuple of polytopes in $\R^n$ defined as $\Delta_i := \conv(\{\unze\}\cup A_i)$ ($i=1,\ldots,n$). Then, for any $f\in K^A$ representing an almost face-generic polynomial map $K^n\to K^n$, it holds that 
\begin{equation}\label{eq:non-properness=crsg}
\cS(f)= \bigcup_{\Gamma \in\sF_{\! D}(\Delta)} \cR_\Gamma(f).
\end{equation} 
\end{theorem}

\begin{proof} We first show the inclusion 
\begin{equation}\label{eq:second-inclusion}
\cS(f)\subset \bigcup_{\substack{\Gamma\prec\Delta \text{ semi-origin}\\
\dim\crsg(f) = n-1}}\crsg(f)
\end{equation} Let $w\in \cS(f)$. Then, the preimage $f^{-1}(w)$ has less than $\mu(f)$ isolated points in $K^n$ counted with multiplicity. Lemma~\ref{lem:preimages} shows that the number of these preimages is less than $\mu(f)=V(\Delta)$. Then, Theorem~\ref{th:Bernstein} shows that $(f-w)_\Gamma=\unze$ has a solution in $\bT$ for some face $\Gamma$ of $\Delta$. This implies that $w\in\crsg(f)$. Lemma~\ref{lem:proj-codim>1} shows that $\Gamma$ is semi-origin. Then, the set $\cS(f)$ is included in the union of all face-resultants corresponding to semi-origin faces of $\Delta$. Those with dimension less than $n-1$ can be omitted as $\cS(f)$ has pure dimension (c.f.~\cite{Jel93}). This yields~\eqref{eq:second-inclusion}.

Now, let $\Gamma\prec\Delta$ be a semi-origin face. We distinguish whether or not $\crsg(f)$ contributes to $\cS(f)$. The rest of the proof is divided into two parts; in \textbf{Part 1)}, we prove the claim
\begin{equation}\label{eq:statement1}
\begin{matrix}
\text{\textit{If $\Gamma$ is dicritical, there is a Zariski open subset $Z\subset\mathcal{R}_\Gamma(f)$}}\\
\text{\textit{  satisfying $Z\subset\cS(f)$.}}
\end{matrix}
\end{equation} Then, in \textbf{Part 2)}, we prove the claim
\begin{equation}\label{eq:statement2}
\begin{matrix}
\text{\textit{If $\Gamma$ is not dicritical and satisfies $\dim\crsg(f) = n-1$,}}\\
\text{\textit{ there is a Zariski open subset $Z\subset\crsg(f)$ satisfying $Z\cap\cS(f)=\emptyset$.}}
\end{matrix}
\end{equation} Since $\cS(f)$ is Zariski closed, the proof then follows from Claims~\eqref{eq:statement1} and~\eqref{eq:statement2}, applied to relation~\eqref{eq:second-inclusion}.

\textbf{Part 1):} 

Assume that $\Gamma$ is dicritical. Then, thanks to Lemma~\ref{lem:face-inclusion_propertuy}, we assume that 
\begin{equation}\label{eq:face-not_included}
\sum\Gamma\not\subset H_1^{\R}\cup\cdots\cup H_n^{\R}.
\end{equation} Note that for any $i\in[n]$, we have $f(H^K_i)=\cR_{\Gamma'}(f)$ for some $\Gamma'\prec\Delta$ satisfying $\sum\Gamma'\subset H^\R_i$. Thus, thanks to~\eqref{eq:face-not_included}, similarly to the proof of Proposition~\ref{prp:face-generic-open}~\ref{it:vg_not_vgprime}, one can check that there is a Zariski open subset $Z\subset\crsg(f)$ that satisfies $Z\cap f(H^K_1\cup\cdots \cup H^K_n)=\emptyset$. Therefore, for any $w\in Z$, all solutions to the system $f-w=\unze$ are located inside $\bT$. Theorem~\ref{th:Bernstein} and Lemma~\ref{lem:preimages} thus show that their number, counted with multiplicities, is less than $\mu(f)$. Then, Claim~\eqref{eq:statement1} follows from Remark~\ref{rem:top-degree-non-prop}.

\textbf{Part 2):} Let $\Kgf$ denote the set of points $w\in\crsg(f)$ outside all other face-resultants in $K^n$, and such that the preimage $f^{-1}(w)$ has only isolated points. That is, we have 
\begin{align*}
\Kgf&:= \crsg(f) \setminus\big( \bigcup_{\substack{\Gamma'\prec\Delta\\ 
\Gamma'\neq\Gamma}}\cR_{\Gamma'}(f)\cup \Pos(f)~\big),
\end{align*} where $\Pos(f):=\{w\in K^n:\dim f^{-1}(w)>0\}$. Since  $f$ is dominant, we have $\dim \Pos(f)\leq n-2$. Then, since $\dim \crsg(f)$ is assumed to be $n-1$, Lemma~\ref{lem:proj-codim>1} shows that $\dim \Kgf= n-1$. 

In what follows, we show that $\Kgf$ is the set $Z$ from Claim~\eqref{eq:statement2}. Since $\Gamma$ is semi-origin and not dicritical by assumption, we have $\sum\Gamma= \sum\Delta\cap H^{\R}_i$ for some $i\in[n]$. Lemma~\ref{lem:resultant_disjoint_from_coordinate} shows that $\crsg(f)=f(H^{K}_i)$. Then, for any $w\in\Kgf$, the system $f-w = \unze$ has isolated solutions in $K^n$, and these are distributed among $\bT$ and $H^{K}_{i}$. As $\VKa((f-w)_{\Gamma'})=\emptyset$ for every $\Gamma'$ different than $\Gamma$,~\cite[Section 5.5, pp. 121]{Ful93} (see also~\cite{rojas1999toric}) shows that the total number of solutions in $K^n$ is equal to the mixed volume of $\Delta$. This in turn implies that $w\not\in \cS(f)$, and thus Claim~\eqref{eq:statement2} follows for $Z:=\Kgf$.
\end{proof}

We recall that a set $S\subset K^n$ is \emph{$K$-uniruled} if for any  $w\in S$, there exists a polynomial map $\varphi:K\to S$ with $\varphi(0) = w$, and $\varphi(K) \subset S$.\\
\begin{corollary}\label{cor:k-uniruled}
The Jelonek set of any almost face-generic polynomial map $f: K^n\to K^n$ is $K$-uniruled.
\end{corollary}

\begin{proof}
We retain the notation for $f$ of Theorem~\ref{thm:algebraic_correspondence}. Thanks to~\eqref{eq:non-properness=crsg}, it is enough to show $K$-uniruledness of $\cR_\Gamma(f)$ for each semi-origin face $\Gamma$ of $\Delta$. If $\Gamma$ is not origin, then $\cR_\Gamma(f)$ is the intersection of $\TKn$ with a set of the form $K^{r}\times X$, where $X$ is an algebraic subset of $K^{n-r}$ (see Definition~\ref{def:r-star-gamma}). Hence, the sets $\cR_\Gamma(f)$ are $K$-uniruled whenever $\Gamma$ is not origin.

Assume now that $\Gamma$ is an origin-face of $\Delta$. Then, Definition~\ref{def:r-star-gamma} shows that $\cR_\Gamma(f)=f_\Gamma(\TKn)$. For every  $w\in\cR_\Gamma(f)$, consider the line $L$ given by the parametrization $z_i(t) = t+p_i$ ($i=1,\ldots,n$), where $f_\Gamma(p_1,\ldots,p_n) = w$. Then, the polynomial map $\varphi:=f_\Gamma(z_1(\cdot),\ldots,z_n(\cdot)):K\to\cS(f)$ is well-defined and satisfies $\varphi(0) = w$. 
\end{proof}

\section{Preliminaries on tropical geometry}\label{sec:prel} We state in this section some well-known facts about tropical geometry (see e.g.~\cite{BIMS15,richter2005first,MS15}, and the references therein). Some of the exposition and notations here are taken from~\cite{Br-deMe11,BB13,EH18}. We will introduce in \S\ref{sub:base-field} some tropical objects such as tropical hypersurfaces, and illustrate their relation to their classical analogues. A key result in this section is Theorem~\ref{thm:mixed_subd} which relates the common intersection of tropical hypersurfaces to decompositions of the Newton polytopes and their faces.
\subsection{The base field}\label{sub:base-field} Let $\Bbbk$ be an algebraically closed field of characteristic zero. A \emph{locally convergent generalized Puiseux series} is a formal series of the form 
\[
c(t)=\underset{r\in R}{\sum} c_rt^r,
\] where $R\subset\mathbb{R}$ is a well-ordered set that is bounded from below, for which all $c_r\in\Bbbk$. We denote by $\K$ the set of all locally convergent generalized Puiseux series. It is an algebraically-closed field of characteristic 0~\cite{markwig2009field}. The field $\K$ can be equipped with the function 
\[
  \begin{array}{lccc}
    \displaystyle \val: & \mathbb{K} & \longrightarrow & \mathbb{R}\cup\{-\infty\} \\
    
    \displaystyle \ & 0 & \longmapsto  & -\infty \\
    
     \displaystyle \ & \underset{r\in R}{\sum} c_rt^r\neq 0 & \longmapsto  & -\min_R\{r\ |\ c_r\neq 0\}.
  \end{array}
\] 
We will call $\val$ the \emph{valuation function} or simply, \emph{valuation}. This extends to a map $\Val:\K^n\rightarrow(\mathbb{R}\cup\{-\infty\})^n$ by evaluating $\val$ coordinate-wise, i.e. $\Val(z_1,\ldots , z_n)=(\val(z_1),\ldots , \val(z_n))$. 

\begin{remark}
In the standard literature (c.f.~\cite{efrat2008valuations}), the valuation of an element in $\K$ is defined as $-\val$. Our choice for it to be the opposite was motivated by some of the conventions in tropical geometry. 
\end{remark} 

Given a polynomial $f\in\K[z_1,\ldots,z_n]$, its \emph{tropical hypersurface}, denoted by $\VT(f)$, is the subset in $\R^n$ defined as $\Val(\VKb (f))$. Tropical hypersurfaces can be determined combinatorially.
Consider the function
\begin{equation}\label{eq:nuf}
  \begin{array}{lccc}
    \displaystyle \nu_f: & \mathbb{Z}^n & \longrightarrow & \mathbb{R}\cup\{-\infty\} \\
    
    \displaystyle \ & a & \longmapsto  &  \begin{cases}
      \val(c_s), & \text{if $ s\in S$,}  \\
      -\infty, & \text{otherwise.}
    \end{cases} \\
    
  \end{array}
\end{equation}

Its \emph{Legendre transform} is a piecewise-linear convex function $ \mathcal{L}(\nu_f): \R^n\to\R$
\[ 
x\mapsto \max_{s\in S}\{\langle x, s\rangle + \nu_f(s)\},
\] where $\langle ~,~\rangle:~\R^n\times\R^n\rightarrow\R$ is the standard Eucledian product. The set of points $p\in\R^n$ at which $\mathcal{L}(\nu_f)$ is not differentiable is called the \emph{corner locus} of $\mathcal{L}(\nu_f)$. We have the famous fundamental theorem of tropical geometry~\cite{MKL06},~\cite[Theorem 3.13]{MS15} that was first announced by Kapranov.\\

\begin{theorem}[Kapranov]\label{th:Kapranov}
The tropical hypersurface $\VT (f)$ of a polynomial $f$ defined over $\K$ is the corner locus of its Legendre transform $\mathcal{L}(\nu_f)$.
\end{theorem}

\subsection{Subdivisions from tropical polynomials}\label{subs:tropical_polynomials} For any two values $a,b\in\mathbb{T}:=\R\cup\{-\infty\}$, their \emph{tropical summation} $a\oplus b$ is defined as their maximum $\max (a,b)$, and their \emph{tropical multiplication} $a \otimes b$ is their usual sum $a+b$. This gives rise to a \emph{tropical semi-field} $(\mathbb{T},\oplus ,\otimes)$, where $\max(a,-\infty) = a$, and $ -\infty + a = -\infty$. A \emph{tropical polynomial} $F$ is defined over the semifield $(\mathbb{T},\oplus ,\otimes)$, which gives rise to a function 
\[
  \begin{array}{lccc}
    F: & \mathbb{T}^n & \longrightarrow & \mathbb{T} \\
    
 	 & x & \longmapsto  &\displaystyle \max_{s\in S}\{\langle x, s\rangle + \gamma_s\}, \\
    
  \end{array}
\] where $S$ is a finite set containing all $s\in\N^n$ for which $\gamma_s\in\R$. The set $S$ is called the \emph{support} of the tropical polynomial $F$, and the linear terms appearing in $F$ are called the \emph{tropical monomials}. 
The \emph{tropicalization} of the polynomial $f$ above is the tropical polynomial 
\[
f^{\trop}(x):=\max_{s\in S}\{\langle x, s\rangle +\val(c_s)\}.
\] This coincides with the piecewise-linear convex function $\mathcal{L}(\nu_f)$ defined above, and thus Theorem~\ref{th:Kapranov} asserts that $\VT(f)$ is the corner locus of $f^{\trop}$. Conversely, the corner locus of any tropical polynomial is a tropical hypersurface.

\subsubsection{Regular polyhedral cell-decompositions}\label{subsub:subdivisions} All polytopes in this paper are assumed to be convex.\\

\begin{definition}\label{def:polyh-subdiv}
Let $\Pi$ be a polytope in $\R^n$. A \emph{polyhedral subdivision} of $\Pi$ is a set of polytopes $\{\pi_i\}_{i\in I}$ satisfying $\cup_{i\in I}\pi_i=\Pi$, and if $i,j\in I$, then $\pi_i\cap\pi_j$ is either empty or it is a common face of $\pi_i$ and $\pi_j$. A polyhedral subdivision $\tau$ of $\Pi$ is called \emph{regular} if there exists a continuous, convex, piecewise-linear function $\varphi:\Pi\rightarrow \R$ such that the polytopes of $\tau$ are exactly the domains of linearity of $\varphi$.
\end{definition}

 Let $\Pi$ be an integer polytope in $\R^n$ and let $\varphi:~\Pi~\cap~\mathbb{Z}^n\rightarrow \R$ be a function. We denote by $\hat{\Pi}(\varphi)$ the convex hull of the set 
\[
\{(s,~\varphi(s))\in\R^{n+1}~|~s\in\Pi\cap\mathbb{Z}^n\}.
\] Then the polyhedral subdivision of $\Pi$ induced by projecting the union of the lower faces of $\hat{\Pi}(\varphi)$ onto the first $n$ coordinates, is regular. \\

\begin{example}\label{ex:subdivision}
The blue part of the second polytope, $\Delta_2$, appearing in Figure~\ref{fig:three_polytopes} is also a polytope. It admits a regular subdivision induced by the function $\varphi$ satisfying $\varphi(1,1,0) = \varphi(1,1,2) = 0$, $\varphi(0,1,2) = 7$, $\varphi(0,2,4) = 3$, $\varphi(1,2,4) = 2$, and $\varphi(2,2,4) = 5$.
\end{example}

\subsubsection{Subdivisions and their duals}\label{subsub:subdivisions2}
Keeping with the same notation as above, the tropical hypersurface $\VT(f)$ is an $(n-1)$-dimensional piecewise-linear polyhedral complex which produces a \emph{polyhedral cell-decomposition} $\Xi$ of $\R^n$. 
This is a finite collection of the relative interiors of polyhedra in $\R^n$, whose closures satisfy Definition~\ref{def:polyh-subdiv}.
An element of $\Xi$ is called \emph{cell}. The $n$-dimensional cells of $\Xi$, are the connected components of the complement of $\VT(f)$ in $\R^n$. 
All together, cells of dimension less than $n$ form the domains of linearity of $f^{\trop}$ at $\VT(f)$. For each such cell $\xi\in\Xi$, the corresponding domain in $\R^n$ is denoted by $|\xi|$; this is the set of points $x\in\R^n$ forming the relatively open polyhedron corresponding to $\xi$. In what follows, we omit the notation $|\cdot|$ whenever it is clear from the context.

The cell-decomposition $\Xi$ induces a regular subdivision $\tau$ of the Newton polytope $\Pi$ of $f$ in the following way (see also~\cite[Section 3]{BB13}). Given a cell $\xi$ of $\VT(f)$ and $x\in\xi$, the set 
\[
\mathcal{I}_{\xi}:=\{s\in S~|~f^{\trop}(x)=\langle x,s\rangle + \val(c_s)\}
\] does not depend on $x$. All together the polyhedra $\delta$, defined as the convex hull of $\mathcal{I}_{\xi}$ form a subdivision $\tau$ of $\Pi$ called the \emph{dual subdivision}, and the polyhedron $\delta$ is called the \emph{dual} of $\xi$. The subdivision $\tau$ is regular as the corresponding function $\varphi$ in Definition~\ref{def:polyh-subdiv} is the same as $\nu_f$ in~\eqref{eq:nuf} (c.f. Example~\ref{ex:subdivision}).

An analogous description holds true for any tuple $f$ of polynomials
 $f_1,\ldots,f_k\in\mathbb{K}[z_1,\ldots,z_n]$. Let $S_1,\ldots,S_k\subset\mathbb{N}^n$, $\Pi_1,\ldots,\Pi_k\subset\R^n$, and $T_1,\ldots,T_k\subset\R^n$ be their respective supports, Newton polytopes, and tropical hypersurfaces respectively. The union of these latter defines a polyhedral cell-decomposition $\Xi$ of $\R^n$. Any non-empty cell of $\Xi$ can be written as 
\[
\xi=\xi_1\cap\cdots\cap \xi_k
\] with $\xi_i\in\Xi_i$ ($i=1,\ldots,k$), where $\Xi_i$ is the polyhedral cell-decomposition of $\R^n$ produced by $T_i$. Any cell $\xi\in\Xi$ can be uniquely written in this way. Similarly, the polyhedral cell-decomposition induces a \emph{mixed dual subdivision} $\tau$  of the Minkowskii sum $\Pi:=\Pi_1 +\cdots+\Pi_k$ in the following way. Any polytope $\delta\in\tau$ is equipped with a unique representation $\delta:=\delta_1+\cdots+\delta_k$ with $\delta_i\in\tau_i$ ($i=1,\ldots,k$), where each $\tau_i$ is the dual subdivision of $\Pi_i$. The above duality-correspondence applied to the (tropical) product of the tropical polynomials gives rise to the following well-known fact (see e.g.~\cite[\S3~\&~4]{BB13}).\\

\begin{theorem}\label{thm:mixed_subd}
There is a one-to-one duality correspondence between $\Xi$ and $\tau$, which reverses the inclusion relations, and such that if $\delta\in\tau$ corresponds to $\xi\in\Xi$, then

\begin{enumerate}[topsep=0mm,itemsep=0mm,label=(\arabic*),ref=(\arabic*)]

 \item\label{it:inters=sum} $\xi=\xi_1\cap\cdots\cap\xi_k$ with $\xi_i\in\Xi_i$ for $i=1,\ldots,n$, then  $\delta = \delta_1 +\cdots+ \delta_k$ with $\delta_i\in\tau_i$ is the polytope dual to $\xi_i$.
 
 \item\label{it:dimens-compl} $\dim\xi + \dim\delta = n$,
 
 \item\label{it:orthogonality} the cell $\xi$ and the polytope $\delta$ span orthonogonal real affine spaces,
 
 \item\label{it:unboundedness} $\delta$ lies on a proper face $F$ of $\Pi$ if and only if the cell $\xi$ is unbounded and contains a half-line spanned by an outward supporting vector to $F$.
\end{enumerate}
\end{theorem} 

\begin{proof}
Item~\ref{it:inters=sum} is a consequence of~\cite[Proposition 4.1]{BB13}. Items~\ref{it:dimens-compl}, and~\ref{it:orthogonality} are obtained by applying item (1) and (2) of~\cite[\S3]{BB13} to the tropical hypersurface $T_1\cup\cdots\cup T_k$. Item~\ref{it:unboundedness}, is similarly obtained from Item (3) of~\cite[\S3]{BB13}, and diagram (3.1) in \textit{loc. cit.} .
\end{proof} We say that the above polyhedral cell-decomposition $\Xi$ is \emph{induced by the tuple $f$}.

\begin{figure}

\centering
\input{fig-subdivision}

\caption{To the left: The subdivisions $\tau_1$, $\tau_2$, and $\tau$ of $\Delta_1$, $\Delta_2$, and $\Delta_1+\Delta_2$ respectively. To the right: the corresponding dual cell-decomposition of $\R^2$. All obtained from Example~\ref{ex:virtual_preimage}}\label{fig:subdivision}
\end{figure}
\subsubsection{Transversal cell-decompositions}\label{subsub:proper} In what follows, we introduce a family of polynomial tuples whose tropical hypersurfaces have transversal intersections.\\

\begin{definition}\label{def:Transv} The cell $\xi\in\Xi$ is \emph{transversal} if $\dim(\delta)=\dim(\delta_1 )+\cdots +\dim(\delta_k)$. A polyhedral cell-decomposition is \emph{transversal} if so are all of its cells.
\end{definition}

\begin{notation}
For any $\xi\in\Xi$ as in Theorem~\ref{thm:mixed_subd}, we will use $\delta(\xi)$ to denote the polytope $\delta$, dual to $\xi$ and we use $\delta(\xi_i)$ ($i=1,\ldots,n$) to denote the polytope $\delta_i$, dual to $\xi_i$.
\end{notation}
Applying~\cite[Theorem 1.1]{OP13} recursively on transversal intersections of tropical hypersurfaces in $\R^n$, we obtain the following generalized version of Theorem~\ref{th:Kapranov} (see also~\cite[~Ch. 4,~\S6,~Theorem 4.6.18]{MS15}).\\
\begin{theorem}\label{thm:Kap-general}
Let $f$ be a tuple of polynomials $f_1,\ldots,f_k\in\K[z_1,\ldots,z_n]$, and assume that $f$ induces a transversal polyhedral cell-decomposition of $\R^n$. Then, it holds that
\begin{equation}\label{eq:kap:gen-1}
\Val (\VKb(f))= \VT(f_1)\cap\cdots\cap \VT(f_k).
\end{equation} 
\end{theorem}

\section{Correspondence for the Jelonek set}\label{sec:correspondence} We start with some definitions.
\begin{definition}[virtual preimage]\label{def:virtual}
Let $L:\R^n\to\R$ be a tropical polynomial function expressed as $x\mapsto\max_{s\in S}\{\langle x, s\rangle + \gamma_s\}$, where $S\subset\N^n\setminus\{\unze\}$ is the support of $L$. For any value $y\in \R\cup\{-\infty\}$, we use the notation $\cT_y(L)$ for the corner locus of the function
\[
x\mapsto\max_{s\in S}\{\langle x, s\rangle + \gamma_s,~y\}.
\] We use $\cT(L)$ as a shorthand for $\cT_{-\infty}(L)$. Similarly, if $L$ is a tropical polynomial map $(L_1,\ldots,L_n):\R^m\to\R^n$, then for any $y:=(y_1,\ldots,y_n)\in(\R\cup\{-\infty\})^n$ we use the notation $\cT_y(L)$ to refer to the intersection $\cT_{y_1}(L_1)\cap\cdots\cap\cT_{y_n}(L_n)$ in $\R^m$. We call $\cT_y(L)$ the \emph{virtual preimage of $y$} under $L$. We use $\cT(L)$ as a shorthand for $\cT_{(-\infty,\ldots,-\infty)}(L)$.
\end{definition}

\begin{example}\label{ex:virtual_preimage}
Assume that $F:\R^2\to\R^2$ is given by 
\begin{equation*}
(F_1,~F_2) = \big( \max (a-1,~2a+b + 2,~3a+2b- 5),~\max (a+b,~2a+2b,~a+2b)
\end{equation*}
 Then, the virtual preimage $\cT_{(-2,-1)}(F)$ is represented in Figure~\ref{fig:subdivision} with the orange line.
\end{example}

 Let $A$ be an $n$-tuple of subsets in $\N^n\setminus\{\unze\}$, and let $\Delta$ be the $n$-tuple of polytopes obtained from $A$ as in Equation~\eqref{eq:delta_0}.\\

\begin{definition}\label{def:tropical_degeneracy}
A polynomial tuple $f\in\K^A$ is called \emph{face-generic} if it is almost face-generic and if for any $\Gamma\prec\Delta$, and any $I\subset [n]$, the tropical intersection $\cap_{i\in I}\VT(f_{i\Gamma})$ is transversal.
\end{definition}

We have the following easy consequence of Proposition~\ref{prp:face-generic-open} and~\cite[Corollary 4.6.11]{MS15}.\\
\begin{lemma}\label{lem:tropical_degeneracy}
Face-generic tuples form a Zariski open subset of $\K^A$.
\end{lemma}

 Consider an element $f\in\K^A$ corresponding to a polynomial map $\KtK$. Let $F:\R^n\to\R^n$ be the tropical polynomial map of $f$. In this section, we show that if $f$ is face-generic, then the tropicalization of its face resultants in the torus coincides with the tropical non-properness set $\TJF$ of $F$. Theorem~\ref{th:main_corresp} will then follow from Theorems~\ref{thm:algebraic_correspondence} and~\ref{thm:tropical-and-tilted}.

\subsection{Dicritical cells and main result} Since the virtual preimage of a tropical polynomial is a tropical hypersurface, for any $y\in\R^n$, the union 
\[
\cT_{y_1}(F_1)\cup \cdots\cup \cT_{y_n}(F_n)
\] forms a polyhedral cell-decomposition of $\R^n$, denoted $\Xi(y)$, and induces a mixed subdivision $\tau_y$ of $\sum\Delta$, dual to $\Xi(y)$ in the sense of Theorem~\ref{thm:mixed_subd} in~\S\ref{subsub:subdivisions2}. Here, the tuple $\Delta$ is the same with respect to $A$ as in~\S\ref{sub:pol-restricted}. Accordingly, the cell-decomposition $\Xi$ of $\R^n$ induced by $F$ is obtained from $\Xi(y)$ by taking $y_i=-\infty$ ($i=1,\ldots,n$). Figure~\ref{fig:subdivision} on the right represents the cell-decomposition $\Xi(-2,-1)$ of $\R^2$ given by the set $\cT_{-2}(F_1)\cup \cT_{-1}(F_2)$, where $F$ is from Example~\ref{ex:virtual_preimage}.\\

\begin{definition}\label{def:dicritical-cells}
A cell $\xi\in\Xi(y)$ is said to \emph{dicritical} if it contains a dicritical half-line (see~\S\ref{sub:prelim}).
\end{definition}

We deduce the following observation from the definitions and Theorem~\ref{thm:mixed_subd}.\\
\begin{lemma}\label{lem:combined-remarks}
Let $\Gamma\prec\Delta$, and let $\xi\in\Xi(y)$, so that $\delta(\xi_i)\subset\Gamma_i$ ($i=1,\ldots,n$). Then, the following holds
	
	\begin{enumerate}[topsep=0mm,itemsep=0mm,label=(\arabic*),ref=(\arabic*)]
		\item\label{it:tilted-dicritical} $\xi$ is dicritical if and only if $\Gamma$ is dicritical, and
		
		\item\label{it:preim=cell} for each $i=1,\ldots,n$ we have $\dim \delta(\xi_i)>0$ if and only if $\xi\subset \cT_y(F_i)$. 
		
	\end{enumerate}		
\end{lemma}
Similarly to Notation~\ref{not:face-restrictions} of~\S\ref{sec:faces}, for any $\Gamma\prec\Delta$, and any $i\in [n]$, we define the function $F_{i\Gamma}:\R^n\to\R$, 
\[
x\mapsto\max_{a\in\Gamma_i}(\langle x, a\rangle + \gamma_a).
\] For any $I\subset [n]$, we define the \emph{restriction of $F$ onto $I\Gamma$} to be the tropical polynomial map
\[
F_{I\Gamma}:=(F_{i\Gamma})_{i\in I}:\R^n\to\R^{\#I}.
\]

\begin{example}\label{ex:subdivision_2}
Consider the polytopes from Figure~\ref{fig:subdivision}, and  $F$ from Example~\ref{ex:virtual_preimage}. Let $\Gamma:=(\conv S_1,\conv S_2)$, where $S_1=\left(\begin{smallmatrix}1&2&3\\0&1&1\end{smallmatrix}\right)$, and $ S_2=\left(\begin{smallmatrix}0&1&2\\0&1&2\end{smallmatrix}\right)$. Then $F_{1\Gamma} = F_1$ and $F_{2\Gamma}=\max (a+b,~2a+2b)$.   
\end{example}

Recall that $\TJF$ is the set of points in $\R^n$ for which $\cT_y(F)$ has a dicritical half-line in $\R^n$, and that $\FDinf$ denotes the set of all dicritical faces $\Gamma\prec\Delta$.\\
\begin{figure}

\centering
\tikzset{every picture/.style={line width=0.75pt}} 

\begin{tikzpicture}[x=0.5pt,y=0.5pt,yscale=-1,xscale=1]

\draw [color={rgb, 255:red, 74; green, 144; blue, 226 }  ,draw opacity=0.35 ][line width=1.5]    (326.01,227.33) -- (298.56,199.67) ;
\draw [color={rgb, 255:red, 0; green, 0; blue, 0 }  ,draw opacity=0.35 ][line width=1.5]    (426.01,277.33) -- (376.01,227.33) ;
\draw  [color={rgb, 255:red, 0; green, 0; blue, 0 }  ,draw opacity=1 ][fill={rgb, 255:red, 0; green, 0; blue, 0 }  ,fill opacity=1 ] (326.01,226.17) .. controls (326.64,226.17) and (327.16,226.69) .. (327.16,227.33) .. controls (327.16,227.97) and (326.64,228.48) .. (326.01,228.48) .. controls (325.37,228.48) and (324.85,227.97) .. (324.85,227.33) .. controls (324.85,226.69) and (325.37,226.17) .. (326.01,226.17) -- cycle ;
\draw [color={rgb, 255:red, 0; green, 0; blue, 0 }  ,draw opacity=0.35 ][line width=1.5]    (376.01,227.33) -- (376.01,177.33) ;
\draw  [color={rgb, 255:red, 0; green, 0; blue, 0 }  ,draw opacity=0.35 ][fill={rgb, 255:red, 0; green, 0; blue, 0 }  ,fill opacity=0.35 ] (376.01,226.17) .. controls (376.64,226.17) and (377.16,226.69) .. (377.16,227.33) .. controls (377.16,227.97) and (376.64,228.48) .. (376.01,228.48) .. controls (375.37,228.48) and (374.85,227.97) .. (374.85,227.33) .. controls (374.85,226.69) and (375.37,226.17) .. (376.01,226.17) -- cycle ;
\draw [color={rgb, 255:red, 74; green, 144; blue, 226 }  ,draw opacity=0.35 ][line width=1.5]    (298.56,199.67) -- (298.56,279.67) ;
\draw [color={rgb, 255:red, 74; green, 144; blue, 226 }  ,draw opacity=0.35 ][line width=1.5]    (272.37,147.37) -- (298.56,199.67) ;
\draw [color={rgb, 255:red, 74; green, 144; blue, 226 }  ,draw opacity=0.35 ][line width=1.5]    (249.1,112.65) -- (272.37,147.37) ;
\draw [color={rgb, 255:red, 74; green, 144; blue, 226 }  ,draw opacity=1 ][line width=1.5]    (402.83,277.83) -- (228.38,103.38) ;
\draw  [color={rgb, 255:red, 74; green, 144; blue, 226 }  ,draw opacity=0.35 ][fill={rgb, 255:red, 74; green, 144; blue, 226 }  ,fill opacity=0.35 ] (272.37,146.21) .. controls (273.01,146.21) and (273.53,146.73) .. (273.53,147.37) .. controls (273.53,148.01) and (273.01,148.53) .. (272.37,148.53) .. controls (271.73,148.53) and (271.21,148.01) .. (271.21,147.37) .. controls (271.21,146.73) and (271.73,146.21) .. (272.37,146.21) -- cycle ;
\draw  [color={rgb, 255:red, 74; green, 144; blue, 226 }  ,draw opacity=0.35 ][fill={rgb, 255:red, 74; green, 144; blue, 226 }  ,fill opacity=0.35 ] (298.56,198.51) .. controls (299.2,198.51) and (299.71,199.03) .. (299.71,199.67) .. controls (299.71,200.31) and (299.2,200.82) .. (298.56,200.82) .. controls (297.92,200.82) and (297.4,200.31) .. (297.4,199.67) .. controls (297.4,199.03) and (297.92,198.51) .. (298.56,198.51) -- cycle ;
\draw [color={rgb, 255:red, 0; green, 0; blue, 0 }  ,draw opacity=0.35 ][line width=1.5]    (326.01,227.33) -- (226.01,177.33) ;
\draw [color={rgb, 255:red, 0; green, 0; blue, 0 }  ,draw opacity=0.35 ][line width=1.5]    (376.01,227.33) -- (326.01,227.33) ;
\draw [color={rgb, 255:red, 245; green, 166; blue, 35 }  ,draw opacity=1 ][line width=2.25]    (376.01,277.33) -- (201.55,102.88) ;
\draw  [color={rgb, 255:red, 245; green, 166; blue, 35 }  ,draw opacity=0.35 ][fill={rgb, 255:red, 245; green, 166; blue, 35 }  ,fill opacity=0.35 ] (352.43,225.05) .. controls (353.82,225.05) and (354.94,226.18) .. (354.94,227.57) .. controls (354.94,228.96) and (353.82,230.09) .. (352.43,230.09) .. controls (351.04,230.09) and (349.91,228.96) .. (349.91,227.57) .. controls (349.91,226.18) and (351.04,225.05) .. (352.43,225.05) -- cycle ;
\draw  [color={rgb, 255:red, 245; green, 166; blue, 35 }  ,draw opacity=0.35 ][fill={rgb, 255:red, 245; green, 166; blue, 35 }  ,fill opacity=0.35 ] (298.87,211.23) .. controls (300.26,211.23) and (301.39,212.36) .. (301.39,213.75) .. controls (301.39,215.14) and (300.26,216.27) .. (298.87,216.27) .. controls (297.48,216.27) and (296.35,215.14) .. (296.35,213.75) .. controls (296.35,212.36) and (297.48,211.23) .. (298.87,211.23) -- cycle ;
\draw  [color={rgb, 255:red, 245; green, 166; blue, 35 }  ,draw opacity=0.35 ][fill={rgb, 255:red, 245; green, 166; blue, 35 }  ,fill opacity=0.35 ] (326.01,225.6) .. controls (326.96,225.6) and (327.73,226.38) .. (327.73,227.33) .. controls (327.73,228.28) and (326.96,229.06) .. (326.01,229.06) .. controls (325.05,229.06) and (324.28,228.28) .. (324.28,227.33) .. controls (324.28,226.38) and (325.05,225.6) .. (326.01,225.6) -- cycle ;
\draw [color={rgb, 255:red, 0; green, 0; blue, 0 }  ,draw opacity=1 ][line width=1.5]    (426.01,277.33) -- (251.55,102.88) ;

\end{tikzpicture}

\caption{The cell-decompositions $\Xi(-2,-1)$ (in transparent colors) and $\Xi_{\Gamma}(-2,-1)$ (in opaque colors) corresponding to Example~\ref{ex:subdivision_2}}\label{fig:subdivision_2}
\end{figure}

\begin{theorem}\label{thm:tropical-and-tilted}
Let $A:=(A_1,\ldots,A_n)$ be a collection of finite subsets in $\N^n\setminus\{\unze\}$, and let $\Delta:=(\Delta_1,\ldots,\Delta_n)$ denote the tuple of polytopes in $\R^n$ defined as $\Delta_i := \conv(\{\unze\}\cup A_i)$ ($i=1,\ldots,n$). Let $f\in \K^A$ representing a face-generic polynomial map $(f_1,\ldots,f_n):\KtK$, and let $F$ denote its tropical polynomial map $(f_1^{\trop},\ldots,f_n^{\trop}):\RtR$. Then, it holds that 
\[
\TJF = \bigcup_{\Gamma\in\FDinf}\Val(\crsg(f)\cap\bT).
\]
\end{theorem}

\begin{proof}
The inclusion ``$\supset$'' does not require face-genericity. Hence, we prove it first. Let $\Gamma$ be a dicritical face of $\Delta$ for which $\crsg(f)$ is not empty. Let $\bm{\circleddash}$ be the subset of indexes $i\in[n]$ for which $\Gamma_i$ contains the origin $\unze$. Let $w$ be a point in $\cR_\Gamma(f)\cap\bT$ and let $y$ be its valuation in $\R^n$. We use $\Sigma_\Gamma(y)$ to denote the cell-decomposition of $\R^n$, induced by the collection of tropical hypersurfaces $S_1,\ldots,S_n$, where 
\[
S_i := \begin{cases}
\cT_{y_i}(F_{i\Gamma}), & \text{if }  i\in\bmcd \\[4pt]
\cT(F_{i\Gamma}),  & \text{otherwise}. \\
\end{cases}
\]
We first show that  $S:=S_1\cap\cdots\cap S_n$ is non-empty. We have $\Trop(\pi_z(X_\Gamma))\subset S$, where $\pi_z$ is the projection $(z,w)\mapsto w$, and $X_\Gamma:=\VKb((f-w)_\Gamma)$ (see Definition~\ref{def:r-star-gamma}). Then, the non-emptyness of $S$ follows from the non-emptyness of $\crsg$.

 The definition of $S$ implies that there exists a cell $\sigma\in\Sigma_\Gamma(y)$ satisfying Lemma~\ref{lem:combined-remarks}~\ref{it:preim=cell}. Since $\Gamma$ is a face of $\Delta$, the cell-decomposition $\Xi(y)$ contains a cell $\xi(y)\subset\sigma$ satisfying $\delta(\xi_i(y))=\delta(\sigma_i)$ ($i=1,\ldots,n$) (e.g. if $\Gamma$ and $F$ are as in Example~\ref{ex:subdivision_2}, then $\Sigma_{\Gamma}(-2,-1)$ is represented in Figure~\ref{fig:subdivision_2}). Then, Lemma~\ref{lem:combined-remarks}~\ref{it:tilted-dicritical} shows that $\xi(y)$ is dicritical, and thus $y\in\TJF$.

Now, we prove the other direction ``$\subset$''. Assume that $\cT_y(F)$ has a dicritical half-line $H\subset\R^n$ for some $y\in\R^n$. Then, Lemma~\ref{lem:combined-remarks}~\ref{it:preim=cell} shows that there exists a cell $\xi\in\Xi(y)$, satisfying $\dim\delta(\xi_i)>0$ ($i=1,\ldots,n$), and $H\subset \xi$. 
Theorem~\ref{thm:mixed_subd}~\ref{it:unboundedness} shows that $\delta(\xi_i)\subset\Gamma_i$ ($i=1,\ldots,n$) for some $\Gamma\prec\Delta$, and Lemma~\ref{lem:combined-remarks}~\ref{it:tilted-dicritical} shows that $\Gamma$ is dicritical. Then, to show that $y\in \Val(\crsg(f)\cap\bT)$, we will construct a point $(z^*,w^*)\in\bT\times \bT$ satisfying $(f(z^*) - w^*)_\Gamma = \unze$, $\Val (z^*) \in H$, and $\Val(w^*) = y$.

The property $\dim\delta(\xi_i)>0$ ($i=1,\ldots,n$) above implies that for any $x\in H$, and any $i\in [n]$, the maximum
\begin{equation}\label{eq:maximum_virtual}
\max_{a\in \Gamma_i\cap A_i}(\langle x, a\rangle + \val(c_a(i)),~y_i)
\end{equation} is reached twice. 

We retain from the beginning of the proof the notation of $\bmcd$ corresponding to the face $\Gamma$. Then, 
we have $(f_i -w_i)_\Gamma = f_{i\Gamma} - w_i$ if $i\in\bmcd$, and $(f_i -w_i)_\Gamma = f_{i\Gamma}$ otherwise. 
Consequently, it holds that 
\begin{equation}\label{eq:cases:omega}
\VT((f_i - w_i)_\Gamma)= \begin{cases}
\VT(f_{i\Gamma} - w_i) & \text{if }  i\in\bmcd  \\[4pt]
\VT(f_{i\Gamma})  & \text{otherwise}. \\
\end{cases}
\end{equation} Let $\bm{\top}\subset [n]$ be the largest subset of indexes satisfying $i\in \bm{\top}$ $\Longrightarrow$ for any $x\in H$, the maximum in~\eqref{eq:maximum_virtual} is not reached at $y_i$. We have that $\bm{\top}$ is well-defined, and contains $[n]\setminus\bmcd $.  Furthermore, the set $\cap_{i\in \bm{\top}}\VT(f_{i\Gamma})$ contains $\xi$ thanks to~\eqref{eq:maximum_virtual}, and is a transversal intersection thanks to Definition~\ref{def:face-generic}~\ref{it:vgi_smooth}. Then, Theorem~\ref{thm:Kap-general} shows that for any $x\in H \subset \cap_{i\in \bm{\top}}\VT(f_{i\Gamma})$, the system
\begin{equation}\label{eq:sys:theta}
f_{i\Gamma} = 0,~i\in \bm{\top}
\end{equation} has a solution $z^*\in\TKn$ satisfying $\Val(z^*)=x$. In fact, since~\eqref{eq:sys:theta} is quasi-homogeneous, there are infinitely-many such solutions. As the map $f$ is face-generic, we may choose $z^*$ so that $f_{i\Gamma}(z^*)\neq 0$ if $i\in [n]\setminus\bm{\top}$.

Finally, we construct $w^*$ as follows. For each $i\in\bm{\top}$, we choose any value for $w_i^*\in \K\setminus\{0\}$ that satisfies $\val(w^*_i) = y_i$. For the remaining indexes $i\in [n]\setminus\bm{\top}$, we set $w_i^*:= f_{i\Gamma}(z^*)$. Since $y_i$ reaches the maximum in~\eqref{eq:maximum_virtual}, and $\Val(z^*)=x$, we get $\val(w^*_i) = y_i$. Ultimately, we obtain $w^*\in\crsg(f)\cap\bT$, and $y=\Val(w^*)$, which finishes the proof.
\end{proof}

\section{Proof of Theorem~\ref{th:main_computation}}\label{sec:computing}
We keep the notations from~\S\ref{sec:correspondence} regarding $A$, $\Delta$, $f\in\K^A$, its tropical map $F$, and the cell-decomposition $\Xi$ of $\R^n$ induced by $F$. Let $\Dicr(F,\xi)$ denote the set of all $y\in\R^n$ for which $\xi$ contains a dicritical half-line from $\cT_y(F)$. We call this a \emph{dicritical image of $\xi$}. Since $\Xi$ is formed by finitely-many disjoint relatively open polyhedra, for any half-line $H\subset\R^n$, there exists a unique cell $\xi\in\Xi$ such that $H\cap\xi$ is a half-line. Then, we obtain the following equation from the definitions:
\begin{equation}\label{eq:trop-non-prop_Dicr}
\TJF = \bigcup_{\xi\in\Xi} \Dicr(F,\xi).
\end{equation}
Recall that we use $\mathcal{V}_\xi(F)$ to denote the set of all $y\in\R^n$ satisfying $\cT_y(f)\cap\xi\neq \emptyset$. In what follows, we will prove Theorem~\ref{th:main_computation} using~\eqref{eq:trop-non-prop_Dicr}; we show that $\Dicr(F,\xi)$ is empty for some types of cells $\xi\in\Xi$ (Lemma~\ref{lem:if-not-contributing}), and for the remaining types, we show that $\Dicr(F,\xi)=\mathcal{V}_\xi(F)$ and that $\overline{\Dicr(F,\xi)}$ is a polyhedron (Lemma~\ref{lem:if-contributing}).

\subsection{Contributing cells for tropical non-properness} Let $\xi\in\Xi$ be a dicritical cell. Then, Items~\ref{it:orthogonality} and~\ref{it:unboundedness} of  Theorem~\ref{thm:mixed_subd} show that its dual cell $\delta(\xi)$ belongs to $\sum\Gamma$ for some dicritical face $\Gamma\prec\Delta$. From the proof of Theorem~\ref{thm:tropical-and-tilted}, we retain the notation $\bmcd$ for the subset of indexes $i\in[n]$ for which $\Gamma_i$ contains $\unze$. Namely, the set $\bmcd$ is empty if $\Gamma$ is not semi-origin, and it is equal to $[n]$ if it is origin. We also define the subset 
\[
\iup:=\{i\in[n]~|~\dim\delta(\xi_i)>0\}.
\] 
\begin{definition}\label{def:suitable}
We say that $\xi$ is \emph{contributing} if $\xi$ is dicritical and either $\bmcd = [n]$ or $[n]\setminus \bmcd\subset\iup$. Namely, for each $i\in [n]$, each polytope $\delta(\xi_i)$ of a contributing cell $\xi$ is the origin $\{\unze\}$ if it has dimension zero.
\end{definition} 

\begin{lemma}\label{lem:if-not-contributing}
If $\xi\in\Xi$ is not a contributing cell. Then, we have $\Dicr(F,\xi)$ is empty.
\end{lemma} 

\begin{proof}
Clearly, the set $\Dicr(F,\xi)$ is empty if $\xi$ is not dicritical (recall Definition~\ref{def:dicritical-cells}). Assume in what follows that $\xi$ is dicritical and that $\xi$ contains a dicritical half-line $H$ from $\cT_{y}(F)$  for some $y\in\R^n$. 

The virtual preimage $\cT_{y}(F)$ induces a cell decomposition $\Xi(y)$ of $\R^n$, where $H\subset \xi(y)$ for some cell $\xi(y)\in\Xi(y)$. Let $\delta(\xi(y)) = \delta(\xi_1(y)) + \cdots + \delta(\xi_n(y))$ denote the dual polytope of $\xi(y)$ in the subdivision $\tau_y$ of $\Delta$ induced by $\cT_y(F)$. Since $H\subset\xi$, Theorem~\ref{thm:mixed_subd}  shows that $\delta(\xi(y))\subset\sum\Gamma$ for some dicritical $\Gamma\prec\Delta$. Then, for each $i\in[n]$ and each $x\in H$, the maximum 
\begin{equation}\label{eq:tropical-withy}
\max_{a\in \delta(\xi_i(y))\cap A_i}(\langle a,~x\rangle + \gamma_{i,a},~y_i)
\end{equation} is reached twice. If for some $i_0\in [n]$, we have $\dim \delta(\xi_{i_0}) =0$, the corresponding expression~\eqref{eq:tropical-withy} equals 
\[
\langle a,~x\rangle + \gamma_{i_0,a}=y_{i_0},
\]  where $a\in\delta(\xi_{i_0})$. Then, from $\delta(\xi_{i_0}(y))\subset \Gamma_{i_0}$, we get $\unze\in\Gamma_{i_0}$, and thus $i_0\in\bmcd$. This shows that $[n]\setminus\bmcd\subset\iup$, and thus $\xi$ is a contributing cell.
\end{proof}

\begin{lemma}\label{lem:if-contributing}
Let $\xi\in\Xi$ be a contributing cell. Then, the set $\mathcal{V}_{\xi}(F)$ satisfies \begin{equation}\label{eq:if-contributing}
\mathcal{V}_{\xi}(F) = \Dicr(F,\xi),
\end{equation} and its closure, $\overline{\mathcal{V}_{\xi}(F)}$, is a polyhedron.
\end{lemma} 

\begin{proof}
For $i=1,\ldots,n$, the tropical polynomial $\hat{F}_i$, given by 
\[
\max_{a\in A_i}(\langle x,~a\rangle+\gamma_{i,a},~y_i)
\] defines a tropical hypersurface $\V(\hat{F}_i):=\VT (f_i - w_i)$ in the space with coordinates $(x,y)\in\R^n\times\R^n$. Let $\mathcal{G}(\hat{F})$ denote the \emph{virtual graph of $F$}, given by the intersection
\[
\mathcal{G}(\hat{F}):=\V(\hat{F}_1)\cap\cdots\cap\V(\hat{F}_n),  
\]and let $\pi_1$ and $\pi_2$ denote the two projections $(x,y)\mapsto x$ and $(x,y)\mapsto y$ respectively. Then, for any $\xi\in\Xi$, we get 
\[
\mathcal{V}_\xi(F) = \pi_2\big(~\pi^{-1}_1(\xi)\cap \mathcal{G}(\hat{F})~\big).
\] Note that $\pi^{-1}_1(\xi)\cap \mathcal{G}(\hat{F})$ coincides with the set of points $(x,y)\in\R^n\times\R^n$ satisfying
\begin{equation}\label{eq:cases-li}
\begin{cases}
L^i_a(x) = L^i_{a'}(x)\geq\max_{c\in A_i} (y_i,~L^i_c(x)) & \text{for all }  i\in\iup \\[4pt]
L^i_b(x) = y_i>\max_{c\in A_i} (L^i_c(x)) & \text{for all }i\not\in\iup, 
\end{cases}
\end{equation} where $a,a',b\in\delta(\xi_i)\cap A_i$, and $L_\alpha^i(x):=\langle\alpha,x\rangle +\gamma_{i,\alpha}$ for any $\alpha\in A_i$. This shows that the closure of $\pi^{-1}_1(\xi)\cap G(\hat{F})$ is a polyhedron, and thus so is $\overline{\mathcal{V}_\xi(F)}$ a polyhedron.

To prove Equality~\eqref{eq:if-contributing}, pick any $y\in\mathcal{V}_\xi(F)$. Then, the set $\xi\cap\cT_y(F)$ is expressed by the equations in~\eqref{eq:cases-li} which, thanks to Theorem~\ref{thm:mixed_subd}~\ref{it:unboundedness}, forms a cone orthogonal to $\sum\Gamma$. Then, by Lemma~\ref{lem:combined-remarks}~\ref{it:tilted-dicritical}, we get that $\xi\cap\cT_y(F)$ contains a dicritical half-line. Therefore, we have $y\in\Dicr(F,\xi)$.\end{proof}

\begin{proof}[Proof of Theorem~\ref{th:main_computation}] It follows from Equation~\ref{eq:trop-non-prop_Dicr}, Lemmas~\ref{lem:if-not-contributing} and~\ref{lem:if-contributing}.
\end{proof}

\section*{Acknowledgments}
The author would like to thank the anonymous referee for their thorough remarks and suggestions on an earlier version of this paper, which led to significant improvements in presentation and correctness.

\subsection*{Contact}\ \\
Boulos El Hilany
\\
Institut f\"ur Analysis und Algebra, TU Braunschweig
\\
{b.el-hilany@tu-braunschweig.de}
\\
\href{https://boulos-elhilany.com}{boulos-elhilany.com}

\bibliographystyle{plain}

\def\cprime{$'$}

\end{document}